\documentclass[twoside,leqno]{article}

\usepackage[letterpaper]{geometry}

\usepackage{siamproceedings,tikz}

\usepackage[T1]{fontenc}
\usepackage{amsfonts}
\usepackage{graphicx}
\usepackage{epstopdf}
\usepackage{enumitem}
\usepackage{algorithmic}
\ifpdf
  \DeclareGraphicsExtensions{.eps,.pdf,.png,.jpg}
\else
  \DeclareGraphicsExtensions{.eps}
\fi

\usetikzlibrary{arrows}
\usetikzlibrary{graphs}
\usetikzlibrary{graphs.standard}
\usetikzlibrary{decorations.markings}

\usepackage{scalerel}\usepackage{setspace}

\newsiamremark{example}{Example}
\newsiamremark{remark}{Remark}
\newsiamremark{hypothesis}{Hypothesis}
\crefname{hypothesis}{Hypothesis}{Hypotheses}
\newsiamthm{claim}{Claim}

\usepackage{amsopn}

\newcommand{\veps}{\epsilon}
\newcommand{\POI}{\mathrm {Poi}}
\newcommand{\cZ}{\mathcal {Z}}
\newcommand{\dZ}{\mathbb {Z}}
\newcommand{\dC}{\mathbb {C}}
\newcommand{\dR}{\mathbb {R}}
\newcommand{\dE}{\mathbb {E}}

\newcommand{\cG}{ \mathbb{G}}
\newcommand{\cGr}{\Dot{\mathcal {G}}}
\newcommand{\cGrr}{\ddot{\mathcal {G}}}
\newcommand{\cGe}{\vec{\mathcal {G}}}

\newcommand{\CAY}{{\mathrm{Cay}}}
\newcommand{\TR}{{\mathrm{Tr}}}
\newcommand{\SCH}{{\mathrm{Sch}}}

\newcommand{\toBS}{\stackrel{{\mathrm{BS}}}{\longrightarrow}}
\newcommand{\toNB}{\stackrel{{\mathrm{NB}}}{\longrightarrow}}
\newcommand{\cP}{\mathcal {P}}
\newcommand{\1}{1\!\!{\sf I}}
\newcommand{\IND}{\1}
\newcommand{\ABS}[1]{{{\left| #1 \right|}}} 
\newcommand{\ee}{\mathrm{e}}

\newcommand{\rootp}{o^{\scaleto{+}{3pt}}}
\newcommand{\rootm}{o^{\scaleto{-}{3pt}}}
\newcommand{\rootpm}{o^{\scaleto{\pm}{3pt}}}

\begin{document}

\title{\Large Sparse graphs and their Benjamini-Schramm limits: a spectral tour}
    \author{Charles Bordenave\thanks{CNRS \& Aix-Marseille Université (\email{charles.bordenave@cnrs.fr)}.}}

\date{}

\maketitle






\begin{abstract} 
Sparse graphs with bounded average degree form a rich class of discrete structures where local geometry strongly influences global behavior. The Benjamini–Schramm (BS) convergence offers a natural framework to describe their asymptotic local structure. In this note, we survey spectral aspects of BS convergence and their applications, with a focus on random Schreier graphs and covering graphs. We review some recent progress on the spectral decomposition of the local operators on graphs. We discuss the behavior of extreme eigenvalues and the growing role of strong convergence in distribution, which rules out spectral outliers.  We also give a new application of strong convergence to the typical graph distance between vertices in Schreier graphs.
 \end{abstract}

\section{Introduction. }

Sparse graphs with bounded average degree form a rich class of discrete structures where local geometry strongly influences global behavior. The Benjamini–Schramm (BS) convergence offers a natural framework to describe their asymptotic local structure through random rooted graphs. The first instances of BS limits were motivated by combinatorial optimization problems on random structures \cite{zbMATH00123709} and by questions on the recurrence of random walks on planar graphs \cite{zbMATH01868576}. Over the years, this notion has become a central component of graph limit theory \cite{zbMATH06122804}. It has proved particularly fruitful for random graphs (such as Erdős–Rényi graphs), for Cayley and Schreier graphs arising from group actions, and for random or deterministic graph coverings. There are now numerous applications and extensions of BS limits to other discrete or geometric structures, including hypergraphs and manifolds.

Since the early applications of BS limits to the study of spectra of local operators on graphs (such as \cite{MR2724665}), many developments and advances have taken place. In this note, we survey some of these recent results, with a special focus on random Schreier graphs and random covering graphs, which provide a rich class of examples.

We begin in Section~\ref{sec:BS} by recalling the definition and basic examples of BS convergence and presenting Schreier graphs and covering graphs. In Section~\ref{sec:op}, we introduce local operators, such as adjacency and non-backtracking operators, and describe the $*$-algebra they generate. We explain how BS convergence translates into convergence in non-commutative distribution and review recent progress on the spectral decomposition of these operators. Section~\ref{sec:edge} focuses on the behavior of extreme eigenvalues and on the growing role of strong convergence in distribution, which rules out spectral outliers. Within this framework, we revisit the celebrated Alon–Boppana bound and present convergence results for extreme eigenvalues of classical random graphs, as well as of random Schreier and covering graphs. We also discuss a new application of strong convergence to the typical graph distance between vertices in Schreier graphs. Along the way, we highlight several open problems that illustrate the richness of the field.

\section{Convergence of sparse marked graphs. }
\label{sec:BS}
In this section, we survey the necessary graph terminology and the  Benjamini-Schramm convergence for sparse graph sequences. We refer to \cite{MR2354165,zbMATH06122804}. 

\subsection{Marked graphs. }\label{subsec:MarkG} A {\em graph} $G = (V,E)$ will be a pair formed by a countable vertex set $V$ and a countable edge set $E$. An edge $e \in E$ has an origin $e_- \in V$ and a destination $e_+ \in V$. The graphs here are undirected meaning that $E$  is equipped with an involution without fixed point $e \mapsto e^{-1}$ such that 
\begin{equation}\label{eq:symme}
(e^{-1})_+ = e_-  \hbox{ for all $e \in  E$}.
\end{equation}
The quotient of $E$ by the involution is the set of unoriented edges of $G$. Our graphs may have multiple edges, that is edges $e,f$ with the same end vertices ($e_\pm = f_\pm$) and self-loops, that is edges $e$ with $e_+ = e_-$.  
A graph is {\em locally finite} if for all $v \in V$, $\deg(v)  < \infty$, where the degree $\deg(v)$ is the number of $e \in E$ such that $e_- = v$ (or  $e_+ = v$ by \eqref{eq:symme}).  A {\em marked graph} $G = (V,E,\xi)$ on a set $\cZ$ is a  graph $(V,E)$ and a map $\xi : E \to \cZ$.


\begin{example}[Conductance graphs] Let $G = (V,E,\xi)$ be a marked graph with a mark function $\xi : E \to  \dR_+$ such that $\xi(e^{-1}) = \xi(e)$ for all $e \in E$. The marks could measure a conductance or strength of an edge.
\end{example}

%
%

\begin{example}[Cayley graphs] \label{ex:cay} Let $\Gamma$ be a finitely generated group and a  generating set $S$ which is symmetric, that is $S^{-1} = S$. The (left) Cayley graph  $\CAY(\Gamma,S)$ is the  $S$-marked graph with vertex set $\Gamma$, edge set $E = S \times \Gamma$ with, if $e = (s,g) \in E$, $e_- = g$, $e^{-1} = (s^{-1}, s.g)$ and mark $\xi(e) = s \in S$. In words, an element $g \in \Gamma$ is connected to $s.g$ with an edge marked by $s$. The degree of a vertex is equal to $|S|$. 
\end{example}

\begin{example}[Schreier graphs] \label{ex:sch} We continue Example \ref{ex:cay} and assume that $\Gamma$ acts on a countable set $V$, that is there exists a representation $\rho : \Gamma \to \mathrm{S}_V$ where $\mathrm{S}_V$ is the symmetric group over $V$. The Schreier graph  $\SCH(\Gamma,S,\rho)$ is the $S$-marked graph with vertex set $V$, edge set $E = S \times V$ with, if $e = (s,x) \in E$, $e_- = x$, $e^{-1} = (s^{-1}, \rho(s)(x))$ and mark $\xi(e) = s \in S$. In words: $x \in V$ is connected to $\rho(s)(x)$ with an edge marked by $s$. 
\end{example}

 \subsection{The local topology. }
 \label{subsec:localtop}
We assume from now on that $\cZ$ is a complete separable metric space. To define a topology on marked graphs, we take a reference vertex.  A {\em rooted marked graph} $g = (G,o)$ is the pair formed by a connected marked graph and a distinguished vertex $o \in V$, the root.  Two rooted marked graphs $(G, o)$ and $(G',o')$ are {\em isomorphic} if there exists a graph isomorphism between $(V,E)$ and $(V',E')$ which sends $o$ to $o'$ and leave the marks invariant.
Let $\cGr = \cGr(\cZ)$ denote the space of isomorphism classes of locally finite rooted marked graphs. In combinatorial language, $\cGr$ is the set of unlabeled locally finite rooted marked graphs. 
It is fine to write for $g \in \cGr$,  $g = (G,o) $ where $G = (V,E,\xi)$ is a marked graph as long as the properties we consider are class properties (that is, do not depend on the choice of the representative of $g$). 

The {\em local topology} on $\cGr$ is the product topology inherited from projections around the root. Concretely, two rooted marked graphs are close if, we can find a large radius $r$ and a small $\delta >0$ such that the subgraphs spanned by balls around the roots are isomorphic and the marks are $\delta$-close.  The set $\cGr$ equipped with the local topology is a complete separable metric space. We denote by $\cP( \cGr)$ the set of probability measures on $\cGr$ equipped with the topology of weak convergence (called the local weak topology). This space  $\cP( \cGr)$ plays a key role in the sequel. An element  $\mu \in \cP( \cGr)$ is the law of a random rooted marked graph $(G,o)$. We will use the probability notation, for an integrable function $f$ in $L^1(\cGr,\mu)$, 
$$
\dE_{\mu} f (G,o) = \int f(G,o) d\mu (G,o). 
$$

\subsection{Benjamini-Schramm convergence. }

Let $G = (V,E,\xi)$ be a marked graph with $V$ finite. The {\em neighborhood distribution} $U(G) \in \cP( \cGr)$ is the law of the unlabeled uniformly rooted marked graph, namely
\begin{equation}\label{eq:defU}
U(G)=\frac{1}{|V|} \sum_{v\in V}\delta_{[G,v]},
\end{equation}
where $| V |$ is the cardinal number of $V$, $[G,v] \in \cGr$ stands for the isomorphism class of $(G(v),v)$ and $G(v)$ is the connected component of $v$ in $G$, see Figure \ref{fi:UG} for an illustration.

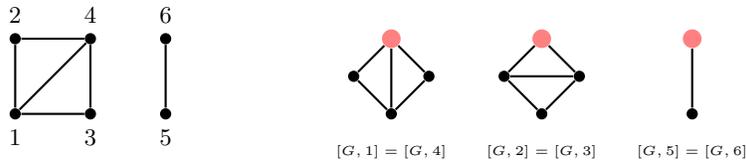
\begin{figure}[htbp]\centering
  \begin{tikzpicture}
    \tikzstyle{real}=[inner sep=1.5pt, fill=black, circle]
    \node[real, label=below:\small{$1$}] (1) at (0,0.5) {};
    \node[real, above of=1, label=\small{$2$}] (2) {};
    \draw[thick] (1)--(2);
    \node[real, label=below:\small{$3$}] (3) at (1,0.5) {};
    \node[real, above of=3, label=\small{$4$}] (4) {};
    \draw[thick] (3)--(4);
    \node[real, label=below:\small{$5$}] (5) at (2,0.5) {};
    \node[real, above of=5, label=\small{$6$}] (6) {};
  
    \draw[thick] (1)--(4);
    \draw[thick] (1)--(3);
    \draw[thick] (2)--(4);
    \draw[thick] (5)--(6);

    \coordinate (shift) at (8,1.5);
    \begin{scope}[shift=(shift)]
     \tikzstyle{real}=[inner sep=1.5pt, fill=black, circle]
      \tikzstyle{root}=[inner sep=2.5pt, fill=red!50, circle]
      \node[root] (1) at (-3,0) {} ;
      \node[real] (2)  at (-3,-1) {};
      \node[real] (3) at (-2.5,-0.5) {};
      \node[real] (4) at (-3.5,-0.5){};
      \draw[thick] (1)--(2);
      \draw[thick] (1)--(3);
      \draw[thick] (1)--(4);
      \draw[thick] (2)--(3);
      \draw[thick] (2)--(4);
      \node at (-3,-1.5) {\tiny{$[G,1] = [G,4]$}};

      \node[root] (01) at (-1,0) {} ;
      \node[real] (02)  at (-1,-1) {};
      \node[real] (03) at (-1.5,-0.5) {};
      \node[real] (04) at (-0.5,-0.5){};
      \draw[thick] (01)--(03);
      \draw[thick] (01)--(04);
      \draw[thick] (03)--(04);
      \draw[thick] (02)--(03);
            \draw[thick] (02)--(04);
       \node at (-1,-1.5) {\tiny{$[G,2] = [G,3]$}};

      \node[root] (5) at (1,0) {};
      \node[real] (6)  at (1,-1) {};
      \draw[thick] (6)--(5);
      \node at (1,-1.5) {\tiny{$[G,5] = [G,6]$}};

    \end{scope}
  \end{tikzpicture}
  \caption{\label{fi:UG} On the left, a graph $G$ with two connected components, on the right, the associated unlabeled rooted graphs. $U(G)$ is the uniform measure on the three displayed unlabeled rooted graphs.}
\end{figure}

\begin{definition} A sequence of finite marked graphs $(G_n)$ converges in {\em Benjamini-Schramm} (BS) sense toward $\mu \in \cP(\cGr)$ if $U(G_n)$ converges weakly to $\mu$. We shall then write $G_n \toBS \mu$ or even $G_n \toBS G$ if $G$ is a random rooted marked graph whose isomorphic class has law $\mu$.
\end{definition}

Graph sequences $G_n$ with BS limits are sparse, in the sense that their typical vertex degree is of order $1$. For example, the sequence $(U(G_n))$ is tight in $\cP( \cGr)$ if the degree sequence is uniformly integrable \cite{zbMATH06423374,MR3405616}.

\subsection{First examples of BS convergence. }

We give a few examples of graphs converging in the BS sense. In  subsections \ref{subsec:schreier}-\ref{subsec:cover}, we will discuss the convergence of finite Schreier and covering graphs.

\begin{example}[Box in $\dZ^d$] \label{exZd} For integer $d \geq 1$, let $\dZ^d_\xi  = (\dZ^d, \xi)$ be the random marked graph where on each unoriented edge of $ \dZ^d$, we put an independent random mark  $\xi(e) = \xi(e^{-1})$ with some common distribution on $\cZ$. Let $\dZ_\xi^d \cap B_n$ be the restriction  $\dZ_\xi^d$ to the box $B_n  = [-n,n]^d$. Then, with probability one  with the respect to the randomness of the marks, we have $\dZ^d_\xi \cap B_n  \toBS \dZ^d_\xi$ (where $\dZ^d_\xi$ is rooted at say $0$). As a byproduct, the bond percolation with parameter $p$ (independently each edge is open with probability $p$ and closed otherwise) on $B_n$ converges toward the connected component of $0$ in  the bond percolation in $\dZ^d$.
\end{example}

\begin{example}[Regular graphs with large girth] Along a subsequence of integers, for integer $d \geq 2$, let $G_n$ be a $d$-regular graph on $n$ vertices: all vertices have degree $d$ ($dn$ is even). If the length of the shortest cycle (ie the girth) of $G_n$  diverges then $G_n \toBS T_d$ where $T_d$ is the infinite rooted $d$-regular tree, see Figure \ref{figGW}. The same conclusion  holds if the graph has few short cycles. More precisely, for integer $k \geq 1$, let $C_k (G_n)$ be the number of cycles of length $k$ in  $G_n$. If for all $k$, $C_k(G_n) / n \to 0$, then we also have $G_n \toBS T_d$. This last condition holds with probability one for a sequence of uniformly sampled $d$-regular (simple) graphs on the vertex set $\{1, \ldots, n \}$. 
\end{example}
\begin{figure}[htb]
\begin{center}
\vspace{-0.4cm}\includegraphics[width=0.4\textwidth]{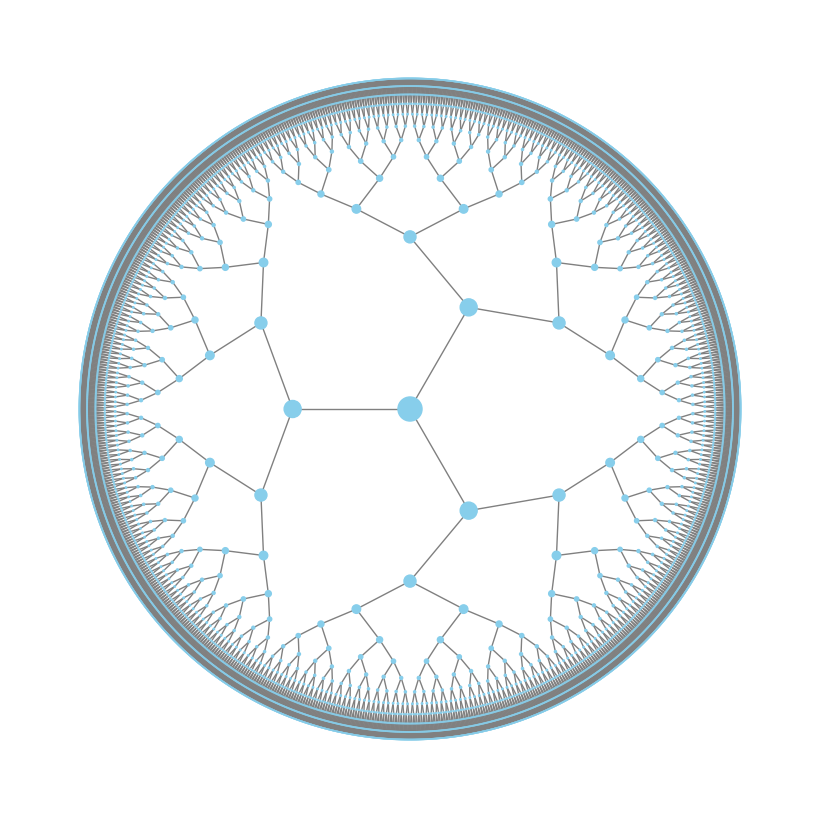} \quad \includegraphics[width=0.4\textwidth]{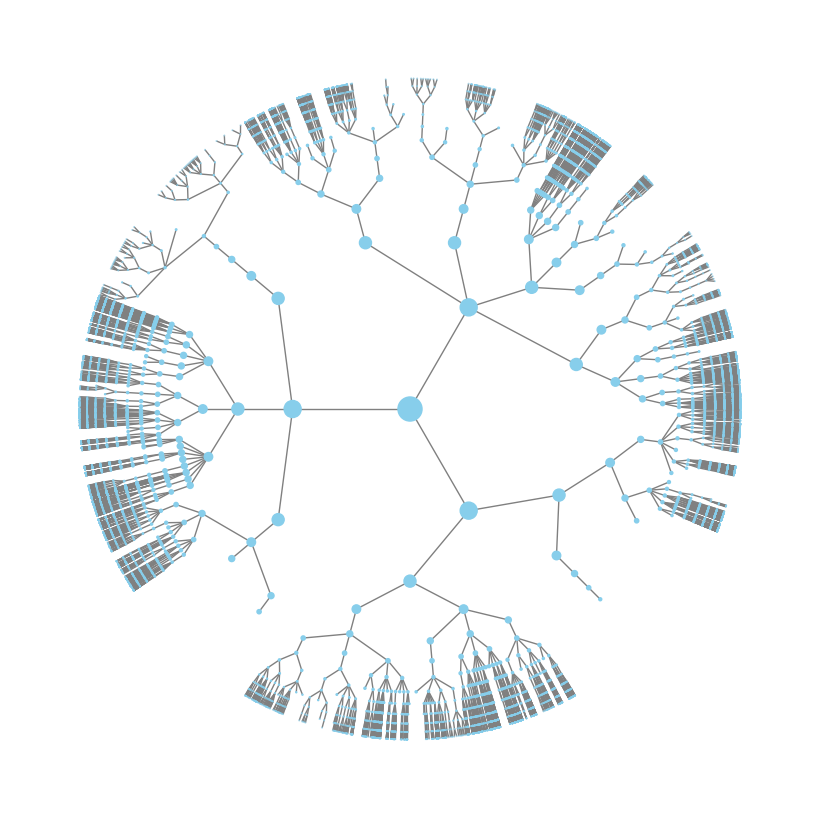}
\label{figGW} \caption{Left: $3$-regular infinite tree. Right: A realization of a Galton-Watson tree with  $\POI(2)$ distribution.}
\end{center}
\end{figure}

\begin{example}[Erd\H{o}s-Rényi random graph and percolation in growing dimension] \label{exER} For real $d >0$, let $G_n$  be the random simple graph on $\{1,\ldots,n\}$ where each undirected edge $\{ u, v \}$, $u >  v$, is present independently with probability $\min(1,d/n)$. Then with probability one, $G_n \toBS T$ where $T$ is a Galton-Watson random tree with offspring distribution $\POI(d)$, the Poisson law with mean $d$. This rooted tree $T$ is built as follows. The root has a number of offsprings sampled according to   $\POI(d)$. Then, independently, each such offspring produces itself a number of offsprings sampled according to   $\POI(d)$ and so on, see Figure \ref{figGW}. The tree $T$ is infinite with positive probability if and only if $d > 1$. Similarly, if $G_n$ is the bond percolation graph on $\dZ^n$ with parameter $d/(2n)$, then  with probability one, as $n \to \infty$, $G_n$ converges locally toward $T$, the same Galton-Watson random tree.
\end{example}

\subsection{Unimodularity. }
The neighborhood distribution $U(G)$ satisfies a reversibility assumption called {\em unimodularity}. An {\em edge-rooted marked graph} $(G,\varrho)$ is a connected marked graph $G$ with a distinguished $\varrho \in E$. The above isomorphisms for rooted marked graphs extend to edge-rooted marked graphs and we may speak of unlabeled edge-rooted marked graphs. We denote by $\cGe$ the set of unlabeled edge-rooted locally finite marked graphs. The local topology extends to this setting as well. A probability measure $\mu \in \cP(\cGr)$ is {\em unimodular} if for every non-negative measurable functions $f$ on $\cGe$, we have  
\begin{equation}
\label{eq:defunimod}\dE_\mu \left[ \sum_{e \in E^-_o}f(G,e)\right]  = \dE_\mu \left[  \sum_{e \in E^+_o}f(G,e) \right] ,   
\end{equation}
where $E^{\pm}_v = \{ e \in E : e_{\pm} = v\}$ is the set of out/in-going edges at $v \in V$.  Property  \eqref{eq:defunimod} first appeared in Aldous \cite{zbMATH00123709}.
 As explained in \cite[Proposition 2.2]{MR2354165}, this is equivalent to the definition of unimodularity introduced in \cite{zbMATH01868576} where it was presented as an extension of unimodularity in group theory. For $\mu$ a Dirac mass at a Cayley graph, Property \eqref{eq:defunimod} boils down to the fact that countable groups are unimodular.

It is easy to check that for any finite marked graph $G = (V,E)$, $U(G)$ is unimodular. Indeed,
$$|V| \cdot \dE_{U(G)} \left[ \sum_{e \in E^-_o}f(G,e)\right] =  \sum_{u \in V}  \sum_{e \in E^-_u}f(G,e)  = \sum_{e \in E} f(G,e)  = \sum_{u \in V}  \sum_{e \in E^+_u}f(G,e) = |V| \cdot \dE_{U(G)} \left[  \sum_{e \in E^+_o}f(G,e) \right].
$$
It is also not hard to check that the set of unimodular measures is closed in $\cP(\cGr)$. Therefore, all BS limits of finite marked graph sequences are unimodular (the converse was recently proven to be false in a breakthrough paper \cite{Bowen:2024azt}). This property is crucial in many applications of BS convergence.

\subsection{BS convergence of Schreier graphs. }

\label{subsec:schreier}

Due to the underlying group action, the BS convergence is somewhat simpler to define for sequence of Schreier graphs.  As in examples \ref{ex:cay}-\ref{ex:sch}, we consider a finitely generated group $\Gamma$ with unit $\ee$. We assume that along a subsequence $n \to \infty$, $\Gamma$  acts on a finite set $V_n$ of cardinal $n$ through a representation $ \rho_n : \Gamma \to \mathrm{S}_{V_n}$. The next folklore lemma asserts that the convergence of Schreier graphs to their Cayley graph is equivalent to the asymptotic triviality of characters $ \TR (  \rho_n(g)  ) =  | \{ x \in V_n : \rho_n(g) (x) = x\} |$.

\begin{lemma} \label{le:BSsch}
The following statements are equivalent: 
\begin{enumerate}[label = (\roman*)]
\item \label{iBS} For some finite symmetric generating set $S$ of $\Gamma$, $\SCH(\Gamma,S,\rho_n) \toBS \CAY(\Gamma,S)$.  
\item \label{iiBS} For all finite symmetric generating sets $S$ of $\Gamma$, $\SCH(\Gamma,S,\rho_n) \toBS \CAY(\Gamma,S)$. 
\item  \label{iiiBS} For all $g \ne \ee$,  $ \frac 1 n  \TR (  \rho_n(g)  )   \to 0
$.
\end{enumerate}
\end{lemma}

Lemma \ref{le:BSsch} notably implies that the convergence of Schreier graphs is the property of the representations $\rho_n$, it does not depend on the specific choice of the generating set. In \eqref{eq:WCGr}, we will see that \ref{iiiBS} is the convergence in distribution of $\rho_n$ toward the regular representation.

\begin{example}[Congruence subgroups]
The group $\dZ$ acts on $\dZ_n$ by congruence mod $n$. Consider the groups $\Gamma = \dZ^d$ or $\Gamma = \mathrm{SL}(d,\dZ)$ for example, if $\rho_n$ is the action by congruence mod $n$ on $\dZ^d_n$ or $\mathrm{SL}(d,\dZ_n)$ then \ref{iiiBS} is easily satisfied. Note that in this case, the Schreier graph is itself be a Cayley graph.
\end{example}

\begin{example}[Random actions of the free group]\label{ex:freeperm}
Let $\mathrm{F}_d$ be the free group with $d$ free generators, say $g_1,\ldots,g_d$. For integer $n \geq 1$, a $d$-tuple $(\sigma_1,\ldots,\sigma_d)$ of permutations in $\mathrm{S}_n$ defines uniquely a permutation representation $\rho_n : \mathrm{F}_d \to  \mathrm{S}_n$ by setting $\rho_n(g_i) = \sigma_i$. By taking independent and uniformly distributed permutations on $\mathrm{S}_n$, we obtain a random permutation of $\mathrm{F}_d $. It was established by Nica \cite{MR1197059} that with probability tending to one as $n \to \infty$, \ref{iiiBS} holds. In particular, for any finite generating set $S = S^{-1}$ of $\mathrm{F_d}$, the random Schreier graph $\SCH(\mathrm{F}_d,S,\rho_n)$ converges toward its associated Cayley graph, in probability. For random actions of free product of groups and surface groups, see \cite{zbMATH07931115} and references therein. 
\end{example}

\subsection{BS convergence of covering graphs. }
\label{subsec:cover}

Let $G = (V,E,\xi)$ be a connected marked graph. A {\em covering graph} or {\em cover} $G_c = (V_c,E_c,\xi_c)$ is a marked graph such that there exists a surjective graph homomorphism $\varphi : E_c \to E$, where by graph homomorphism, we mean a map which preserves marks ($\xi_c(e) = \xi(\varphi(e)$), adjacency  (if $e_- = f_-$ then $\varphi(e)_- = \varphi(f)_-$)  and inversion ($\varphi(e^{-1}) = \varphi(e)^{-1}$). This map is called the {\em covering map}.

If $G_c$ is connected then the map $ e \mapsto |\varphi^{-1} (e) |$ is constant (could be infinite). For integer $n \geq 1$, a $n$-cover of $G$ is a cover $G_c$ such that the covering map satisfies $|\varphi^{-1} (e) | = n$ for all $e \in E$, see Figure \ref{fig2}. The {\em universal covering tree} $T_G$ of $G$ is a cover which is a tree, it is unique up to graph isomorphisms. The universal cover is maximal, in the sense that any connected cover of $G$ is also covered by $T_G$.

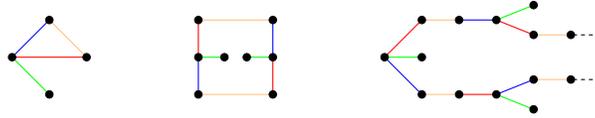
\begin{figure}[htb]
\begin{center}  
\resizebox{8cm}{!}{
\begin{tikzpicture}[main/.style={circle, draw , fill = black!15, text = black, thick},high/.style={circle, draw , fill = red!90, text = black, thick}]


\tikzset{->-/.style={decoration={
  markings,
  mark=at position .5 with {\arrow{>}}},postaction={decorate}}}

\draw[thick,color = red] (0,0) to (2,0)  ; 
\draw[thick, color= blue] (0,0) to (1,1)  ;
\draw[thick, color = green] (0,0) to (1,-1)  ; 
\draw[thick, color = orange!50] (1,1) to (2,0)  ; 

\draw[fill = black] (0,0) circle(0.1)  ; 
\draw[fill = black] (2,0) circle(0.1)  ; 
\draw[fill = black] (1,1) circle(0.1)  ; 
\draw[fill = black] (1,-1) circle(0.1)  ;

\coordinate (shift) at (5,0);
\begin{scope}[shift=(shift)]
\draw[thick,color = red] (0,0) to (0,1)  ; 
\draw[thick,color = orange!50] (0,1) to (2,1)  ;
\draw[thick,color= blue] (2,1) to (2,0)  ; 
\draw[thick,color = red] (2,0) to (2,-1)  ; 
\draw[thick,color = orange!50] (2,-1) to (0,-1)  ; 
\draw[thick,color= blue] (0,0) to (0,-1)  ;

\draw[thick,color =green] (0,0) to (0.7,0)  ; 
\draw[thick,color =green] (2,0) to (1.3,0)  ;

\draw [fill = black]  (0,0) circle(0.1)  ; 
\draw[fill = black]   (2,1) circle(0.1)  ; 
\draw[fill = black]   (0,1) circle(0.1)  ; 
\draw [fill = black]  (0.7,0) circle(0.1)  ; 

\draw[fill = black]  (2,0) circle(0.1)  ; 
\draw[fill = black]  (0,-1) circle(0.1)  ; 
\draw[fill = black] (2,-1) circle(0.1)  ; 
\draw[fill = black]  (1.3,0) circle(0.1)  ; 
\end{scope}

    \coordinate (shift) at (10,0);
    \begin{scope}[shift=(shift)]

\draw[thick,color = red] (0,0) to (1,1)  ;
\draw[thick,color =green] (0,0) to (1,0)  ; 
\draw[thick,color= blue] (0,0) to (1,-1)  ; 

\draw[thick,color = orange!50] (1,1) to (2,1)  ;
\draw[thick,color = orange!50] (2,-1) to (1,-1)  ; 

\draw[thick,color= blue] (2,1) to (3,1)  ;
\draw[thick,color = red] (2,-1) to (3,-1)  ; 

\draw[thick,color =green] (4,1.4) to (3,1)  ;
\draw[thick,color= red] (4,0.6) to (3,1)  ;
\draw[thick,color= blue] (4,-0.6) to (3,-1)  ;
\draw[thick,color =green] (4,-1.4) to (3,-1)  ;

\draw[thick,color = orange!50] (4,0.6) to (5,0.6)  ;
\draw[thick,color = orange!50] (4,-0.6) to (5,-0.6)  ;
strictly less symbol
\draw[thick,dashed] (5.8,0.6) to (5,0.6)  ;
\draw[thick,dashed] (5.8,-0.6) to (5,-0.6)  ;

\draw[fill =  black] (0,0) circle(0.1)  ; 
\draw[fill =  black] (1,-1) circle(0.1)  ; 
\draw[fill = black] (1,1) circle(0.1)  ; 
\draw[fill =  black] (1,0) circle(0.1)  ; 

\draw[fill = black] (2,-1) circle(0.1)  ; 
\draw[fill =  black] (2,1) circle(0.1)  ; 

\draw[fill =  black] (3,-1) circle(0.1)  ; 
\draw[fill =  black] (3,1) circle(0.1)  ; 

\draw[fill =  black] (4,-1.4) circle(0.1)  ; 
\draw[fill =  black] (4,-0.6) circle(0.1)  ; 

\draw[fill =  black] (4,1.4) circle(0.1)  ; 
\draw[fill = black] (4,0.6) circle(0.1)  ;

\draw[fill =  black] (5,0.6) circle(0.1)  ;  

\draw[fill = black] (5,-0.6) circle(0.1)  ; 

\end{scope}

\end{tikzpicture}
} 
\end{center}
\label{fig2} \caption{A graph, a $2$-lift and their universal covering tree.}\end{figure}

Covers can be understood through the lens of group actions. Assume that $G$ is finite and write $E  = \{e_1, \ldots, e_d, e_1^{-1}, \ldots, e_d^{-1} \}$. Up to isomorphisms, it is possible to represent any $n$-cover of  $G$ thanks to a $d$-tuple of permutations $(\sigma_1, \ldots , \sigma_d )$ in $\mathrm{S}_n$, see Figure \ref{fi:lift}. The vertex set is the fiber $V_n = V \times \{1,\ldots, n\}$, the edge set $E_n = E \times \{1,\ldots, n\}$ where the edge $(e_i,x)$ satisfies $(e_i,x)_- = ((e_i)_-,x)$ and $ (e_i,x)^{-1} = (e_i^{-1}, \sigma_{i}(x))$. Hence as explained in Example \ref{ex:freeperm}, the $n$-cover is encoded by the representation $\rho_n : \mathrm{F}_d \to \mathrm{S}_n$ satisfying $\rho_n(g_i) = \sigma_i$.   Similarly, $|V|$ disjoint copies of the universal cover of $G$ are obtained by considering the marked graph on the vertex set $V_{\mathrm F} = V \times \mathrm F_d$, edge set $E_{\mathrm F} = E \times \mathrm F_d$ and the edge $(e_i,g)$ satisfies $(e_i,g)_- = ((e_i)_-,g)$ and $ (e_i,g)^{-1} = (e_i^{-1}, g_{i} g)$.

\begin{figure}[htbp]\centering
  \begin{tikzpicture}
    \tikzstyle{real}=[inner sep=1.5pt, fill=black, circle]
    \tikzstyle{fib}=[inner sep=1.5pt, fill=black, circle]

    \node[real, label=right:\small{$b$}] (1) at (0,0) {};
    \node[real,label=above:\small{$a$}] (2) at (0,1) {};
    \node[real,label=below:\small{$c$}] (3) at (-.5,-1) {};
    \draw[thick,color = red] (1)--(2);
    \draw[thick,color = blue] (1)--(3);
    \draw[thick,color = green] (3)--(2);

 \coordinate (shift) at (3.5,0);
    \begin{scope}[shift=(shift)]
\node[fib,label=right:\small{$b1$}] (11) at (0,0) {};
    \node[fib,label=above:\small{$a1$}] (21) at (0,1) {};
    \node[fib,label=below:\small{$c1$}] (31) at (-.5,-1) {};

\end{scope}

 \coordinate (shift) at (6,0);
    \begin{scope}[shift=(shift)]
\node[fib,label=right:\small{$b2$}] (12) at (0,0) {};
    \node[fib,label=above:\small{$a2$}] (22) at (0,1) {};
    \node[fib,label=below:\small{$c2$}] (32) at (-.5,-1) {};

\end{scope} 

 \coordinate (shift) at (8.5,0);
    \begin{scope}[shift=(shift)]
\node[fib,label=right:\small{$b3$}] (13) at (0,0) {};
    \node[fib,label=above:\small{$a3$}] (23) at (0,1) {};
    \node[fib,,label=below:\small{$c3$}] (33) at (-.5,-1) {};

\end{scope} 

    \draw[thick,color = red] (11)--(21);
    \draw[thick ,color = blue] (11)--(32);
    \draw[thick ,color = green] (31)--(21);

    \draw[thick ,color = red] (12)--(23);
    \draw[thick ,color = blue] (12)--(33);
    \draw[thick ,color = green] (32)--(23);

    \draw[thick ,color = red] (13)--(22);
    \draw[thick ,color = blue] (13)--(31);
    \draw[thick ,color = green] (33)--(22);

    \draw[thick,dashed,color = black!30] (1)--(13);
\draw[thick,dashed,color = black!30] (2)--(23);
\draw[thick,dashed,color = black!30] (3)--(33);

  \end{tikzpicture}
  \caption{\label{fi:lift} A base graph on the left and a $3$-cover on the right. Each vertex of the base graph has a discrete fiber of $3$ elements. The permutation of each edge encodes how elements are connected. }
\end{figure}
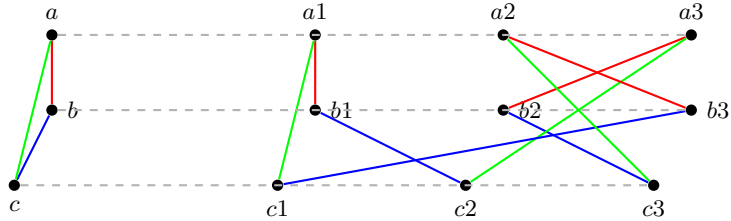

These constructions generalize the Schreier and Cayley graphs of $\mathrm{F}_d$ associated to the free generators and their inverses $S = \{g_1,\ldots,g_d, g^{-1}_1,\ldots, g_d ^{-1}\}$ (consider a graph $G$ with a single vertex and $d$ loop edges). We may also replace $\mathrm{F}_d$ by any finitely generated group $\Gamma$, a representation $\rho_n : \Gamma \to \mathrm{S}_n$ and a symmetric generating set $S$ in bijection with the edge set $E$ of the marked graph $G$. The maximal Abelian cover is for example obtained by considering $\Gamma = \dZ^d$ and its natural generators. Let us call $(\Gamma,S)$-cover of $G$ and $(\Gamma,S,\rho_n)$-cover of $G$ the corresponding graphs.  As a continuation of Lemma \ref{le:BSsch}, we have the following.

\begin{lemma} \label{le:BSsch2}
Let $G = (V,E,\xi)$ be a finite connected marked graph with a $(\Gamma,S)$-cover and $\rho_n : \Gamma \to \mathrm{S}_n$ a subsequence of permutations representation of $\Gamma$. Statements \ref{iBS}-\ref{iiiBS} in Lemma \ref{le:BSsch} imply that
\begin{enumerate}[label = (\roman*)]
\setcounter{enumi}{3}
\item \label{ivBS} The $(\Gamma,S,\rho_n)$-cover of $G$ converges in the BS sense toward the (connected component) of the $(\Gamma,S)$-cover of $G$ rooted at $(o,\ee)$ where $o$ is uniform on $V$.  
\end{enumerate}
\end{lemma}

Hence, as discussed in Example \ref{ex:freeperm}, for any finite connected marked graph $G = (V,E,\xi)$, in probability, a uniformly sampled $n$-cover of $G$ BS-converges toward its universal cover $T_G$ rooted uniformly.

\section{Operators on unimodular graphs. }

\label{sec:op}

We now introduce operators defined on marked graphs. BS convergence offers a natural framework to study spectral properties of large graphs and their limits.

\subsection{Local operators for Schreier and covering graphs. }

As above let $\Gamma$ be a finitely generated group with unit $\ee$. 
Recall that the {\em group algebra} $\mathbb{C}[\Gamma]$ is the set of formal finite weighted sums of group elements: 
$$
a = \sum_{g \in \Gamma} a_g g ,
$$
where the support of $a$ is finite and $a_g \in \mathbb{C}$.  $\mathbb{C}[\Gamma]$ defines a unital $*$-algebra with unit $\ee$ by setting, for $a,b \in \mathbb{C}[\Gamma]$
$
a ^* = \sum_g \overline{a_g} g^{-1}$ 	and  $a.b  = \sum_{g,h} a_g b_h  (gh)$. 
The left-regular representation $\lambda : \Gamma \to \mathcal{B} (\ell^2(\Gamma))$ is defined by, for $g,h \in \Gamma$,  $\lambda(g) \delta_h = \delta_{gh}$ where $\delta_g \in \ell^2 (\Gamma)$ is the Dirac delta function.  This is a faithful unitary representation. For $a \in \mathbb{C}[\Gamma]$, the operator $\lambda(a) \in \mathcal{B}(\ell^2(\Gamma))$ is \begin{equation}\label{eq:deflambdaa}
\lambda(a) = \sum_g a_g \lambda(g).
\end{equation}

 If $S$ is the union of supports of $a$ and $a^*$ minus $\ee$, the operator $\lambda(a)$ can be interpreted as a local operator on $\CAY(\Gamma,S)$ where the edge $e =(s,g) \in S \times \Gamma $ receives a weight $a_s$. For example, if $a = \IND_S = \sum_{g \in S} g$ then $\lambda(a)$ is the adjacency operator of $\CAY(\Gamma,S)$. If $a  \in \dR_+ [\Gamma]$ and $\sum_g a_g = 1$, then $\lambda(a)$ is the transition kernel of the random walk on $\CAY(\Gamma,S)$ where the random walker jumps from $g$ to $s.g$ with probability $a_s$.

Similarly, if $\rho : \Gamma \to \mathrm{S}_V$ is a permutation representation then 
$$
\rho(a) =  \sum_g a_g \rho(g) \in \mathcal B (\ell^2 (V))
$$
is the corresponding local operator on $\SCH(\Gamma,S,\rho)$. We have that $\rho(a)$ and $\lambda(a)$ are self-adjoints when $a = a^*$.

The positive faithful state $\tau$ on $\mathbb{C}[\Gamma]$ is $\tau(a) = a_{\ee} = \langle \delta_{\ee} , \lambda(a) \delta_{\ee} \rangle$. If along a subsequence of integers, $\rho_n$ is a permutation representation on $V_n =  \{1,\ldots,n\}$, then we can reinterpret Lemma \ref{le:BSsch}\ref{iiiBS} as the {\em convergence in distribution} of $\rho_n$ toward $\lambda$:
\begin{equation}\label{eq:WCGr}
\lim_{n\to \infty} \frac {1} {n} \TR (\rho_n(a) ) = \tau(a), \quad \hbox{  for all $a \in \dC[\Gamma]$.} 
\end{equation}

Let $G = (V,E,\xi)$ be a finite marked graph. Let $S$ be a symmetric set of generators of $\Gamma$ and isomorphic to $E$, that is, there is a bijection  $\iota : S \to E$ such that $\iota(s^{-1}) = \iota(s)^{-1}$. Local operators on the $(\Gamma,S)$-cover of $G$ and the $(\Gamma,S,\rho_n)$-cover of $G$ can also be written in the group algebra as follows. For $u,v \in V$, let $E_{uv} \in M_V(\dC)$ be the matrix where all entries are $0$ but entry $(u,v)$ equal to one. Then, the adjacency operator of $G$ can be written as 
$$
A = \sum_{g \in S} E_{v_g u_g},
$$
where if $g \in S$ and $e  = \iota (g) \in E$ we have set $u_g = e_-$, $v_g = e_+$.  
If $\rho_1$ is the trivial representation $\rho_1 : g \mapsto 1$ of dimension $1$, we could thus write 
$
A = \rho_1(\IND_E),
$
where 
$$
\IND_E = \sum_{g \in S} E_{v_g u_g} g \;  \in M_V(\dC) [\Gamma].
$$
Similarly, the adjacency operator of the $ (\Gamma,S,\rho_n)$-cover of $G$ is
\begin{equation}\label{eq:AN}
A_n = \rho_n( \IND_E ) = \sum_{g \in S} E_{v_g u_g} \otimes \rho_n( g)  \; \in M_V (\mathbb{C}) \otimes M_n ( \mathbb C)
\end{equation}
Importantly, the operator $A_n$ restricted to $\mathbb C^V \otimes \mathrm 1$ coincides with $A$. The space $\mathbb C^V \otimes \mathrm 1$ is the set of vectors that are constant in each fiber above the vertices of $G$.  If $\lambda$ is the left regular representation of $\Gamma$ then 
\begin{equation}\label{eq:AF}
\lambda( \IND_E ) = \sum_{g \in S}  E_{v_g u_g} \otimes \lambda(g)  \; \in M_V (\mathbb{C}) \otimes \mathcal B (\ell^2 (\Gamma)),
\end{equation}
is the adjacency operator of the $(\Gamma,S)$-cover of $G$. As explained above, if $\Gamma  = \mathrm{F}_d$ and $S$ the free generators, $\lambda( \IND_E )$ is the adjacency operator of $|V|$ copies of the universal cover of $G$. More generally, for $a \in M_V(\dC) [\Gamma]$, we set
\begin{equation}\label{eq:ANF}
\rho_n(a) = \sum_{g} a_g \otimes \rho_n(g) \quad \hbox{ and } \quad \lambda(a) = \sum_{g} a_g \otimes \lambda(g), 
\end{equation}
could describe a local operator on the $ (\Gamma,S,\rho_n)$-cover or the $ (\Gamma,S)$-cover of $G$, see also \cite{zbMATH01421105,9317918}. We note finally that $\rho_n(a)$ and $\lambda(a)$ are self-adjoint when $a = a^* = \sum_g a^*_g g^{-1}$.

\subsection{Local operators: general case. }

As explained in \cite{MR2354165,MR2593624}, the above unital $*$-algebra of operators equipped with its faithful normal tracial state has an analog for unimodular measures $\mu \in \cP(\cGr)$. We will not give here the precise definition here. Let us simply mention that it contains the following family of {\em local operators}.

We introduce a slight variation of the sets $\cGr$ and $\cGe$ and let $\cGrr$ be the set of  unlabeled locally finite connected marked graphs with two ordered root vertices $(\rootm,\rootp)$. We equip  $\cGrr$ with the local topology. Let $a : \cGrr \to \dC$ be a continuous function with bounded support in the sense that, for some $h \geq 0$, $a (G,\rootm,\rootp) = 0$  whenever the graph distance between $\rootm$ and $\rootp$ is larger than $h$. Let $G = (V,E,\xi)$ be a locally finite marked graph.  We define the operator $A = A_{G,a}$ for all compactly supported $\varphi \in \ell^2 (V)$ by: 
\begin{equation}\label{eq:defAloc}
A \varphi   (v) = \sum_{u \in V} a (G^{(uv)}) \varphi(u) \quad \hbox{ for all $v \in V$},
\end{equation}
where by convention $a(G^{(vu)}) = 0$ if $u$ and $v$ are not in the same connected component and otherwise $G^{(uv)} \in \cGrr$ is the unlabeled marked graph of the connected component of $u$ and $v$ with roots $(u,v)$. For example, if the mark space is $\dC$ and
\begin{equation}\label{eq:adj}
a (G, \rootm,\rootp) = \sum_{e \in E: e_{\pm} = \rootpm} \xi(e)
\end{equation}
 then $A$ is the {\em weighted adjacency operator} of $G$. In the above group case, by transitivity, we can identify the operator $A$ in \eqref{eq:defAloc} with $\lambda(a)$ for some $a \in \dC[\Gamma]$.  Note that for $\varphi \in \ell^2 (V)$ compactly supported,  $A \varphi  $ is indeed in $\ell^2(V)$ because $G$ is locally finite. If the degree of $G$ is not uniformly bounded, $A$ may however be an unbounded operator (the adjacency operator for example). Some care is then  necessary to define the domain  of $A$. Note that $A$ is symmetric if $a$ is symmetric, that is if for all $(G,\rootm,\rootp) \in \cGrr$, we have $a (G,\rootm,\rootp ) = \overline a (G,\rootp,\rootm)$.

When $\mu \in \cP(\cGr)$ is unimodular and $(G,o)$ has distribution $\mu$, then the above function $a : \cGrr \to \dC$ is an element (or rather affiliated) to the von Neumann algebra defined in \cite{MR2354165,MR2593624} whose state is $\tau(a) = \dE_{\mu} \langle \delta_o , A_{G,a} \delta_o \rangle$. A first consequence of unimodularity is the following, see \cite[Proposition 2.2]{zbMATH06902684} (written for the adjacency operator).
\begin{lemma}\label{le:nelson}
Let $\mu \in \cP(\cGr)$ be unimodular and $(G,o)$ with distribution $\mu$. Let $a : \cGrr \to \dC$ be a symmetric continuous function with bounded support. Then with $\mu$-probability one, $A_{G,a}$ is essentially self-adjoint (that is, has a unique self-adjoint extension). 
\end{lemma}  

Thanks to this result, we can inquire about the spectral decomposition of symmetric local operators for unimodular random rooted marked graphs. 

\subsection{BS Continuity of the average spectral measure. }

Fix a symmetric continuous function $a : \cGrr \to \dC$ with bounded support. Let $G = (V,E,\xi)$ be a locally finite marked graph. Assume that $A = A_{G,a}$ defined by \eqref{eq:defAloc} extends to a self-adjoint operator with domain $D(A) \subset \ell^2(V)$. Recall from the spectral theorem that for any $\varphi \in D(A)$ of unit norm, the exists a probability measure $m_{A}^\varphi \in \cP(\dR)$, the {\em spectral measure at vector $\varphi$} such that for any bounded continuous functions $f$, 
$$
\langle \varphi , f(A)\varphi \rangle = \int f (\lambda) d m_A ^{\varphi} (\lambda).
$$ 
If $E$ is the resolution of the identity for $A$, then $ d m_A ^{\varphi} (\lambda) = \langle \varphi , dE (\lambda)\varphi \rangle$. If $G$ is finite then 
$m_A^{\varphi} = \sum_k \delta_{\lambda_k} | \langle\varphi , \psi_k \rangle |^2, $
where $(\lambda_k)$ are the eigenvalues (counting multiplicities) and $(\psi_k)$ is orthornormal basis of eigenvectors of $A$ and we denote the empirical distribution of eigenvalues by 
$$
m_A = \frac{1}{|V|} \sum_{k = 1}^{|V|} \delta_{\lambda_k}.
$$ 
The spatial average of spectral measures is the empirical distribution of eigenvalues: 
$$
\dE_{U(G)} [  m_A^{\delta_o} ] = \frac{1}{|V|}  \sum_{v \in V}  m_A^{\delta_v} = m_A. 
$$
More generally, if $\mu \in \cP(\cGr)$ is unimodular, in view of Lemma \ref{le:nelson}, we can define the {\em average spectral measure} as 
$$
m_{\mu,a} = \dE_{\mu} [ m^{\delta_o}_{A_{G,a}} ].
$$
In particular, for finite $G$, $m_{U(G),a}  = m_{A_{G,a}}$. In the group case $\Gamma$ with unit $\ee$ and symmetric $a \in \mathbb C[\Gamma]$, we simply write $m_a$ for the spectral measure $m_{\lambda(a)}^{\delta_{\ee}}$. 
The following theorem asserts that the average spectral measure is continuous for the BS convergence.
\begin{theorem}\label{th:BScontspec}
Let $a : \cGrr \to \dC$ be a symmetric continuous function  with bounded support. Let $(G_n)$ be sequence of finite marked graph such that $(G_n) \toBS \mu \in \cP(\cGr)$. Then, for the weak topology on $\cP(\dR)$,
$
m_{A_{G_n,a} } \to m_{\mu,a}.
$
Moreover, if $a$ takes values in $\dZ$, we have for any $\lambda \in \dR$, 
$
 m_{A_{G_n,a} } (\{\lambda \} ) \to  m_{\mu,a}( \{\lambda \}).
$
\end{theorem}
See Figure \ref{figERREG} for an illustration. A proof of the first claim of this theorem can be found in \cite{zbMATH06902684,bordenave2024largedeviationsmacroscopicobservables} (with slightly different assumptions). It relies on the observation that $(G,v) \mapsto m_{A_{G,v}}^{\delta_v}$ is continuous for the local topology when $G$ has uniformly bounded degrees (since $\int \lambda^k dm_{A_{G,a}}^{\delta_v}(\lambda) = \langle \delta_v , A_{G,a}^k \delta_v \rangle$  is function of a finite neighborhood of $(G,v)$). The second claim is an instance of L\"uck approximation theorem due in this context to \cite{Abrt2014BenjaminiSchrammCA} (see \cite{zbMATH06902684} for graphs with non-uniformly bounded degrees). It has the striking consequence that if $(G_n) \toBS \mu \in \cP(\cGr)$ and $a$ takes values in $\dZ$ then $m_{\mu,a}(\{\lambda \}) >0$ implies that $\lambda$ is a totally-real algebraic integer.

\begin{figure}[htb]
\begin{center}
\includegraphics[height=4cm]{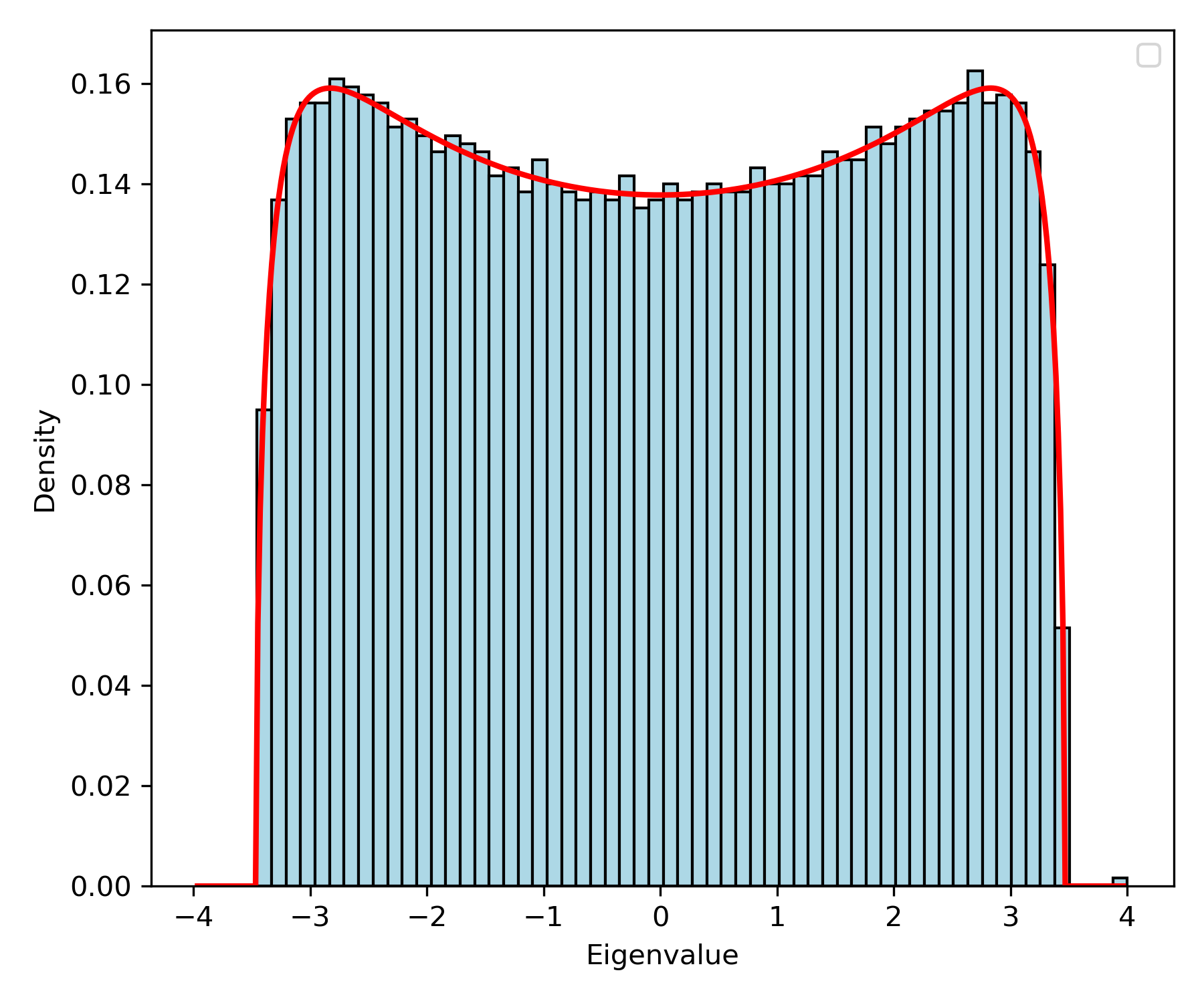} \hspace{2.5cm}\includegraphics[height=4cm]{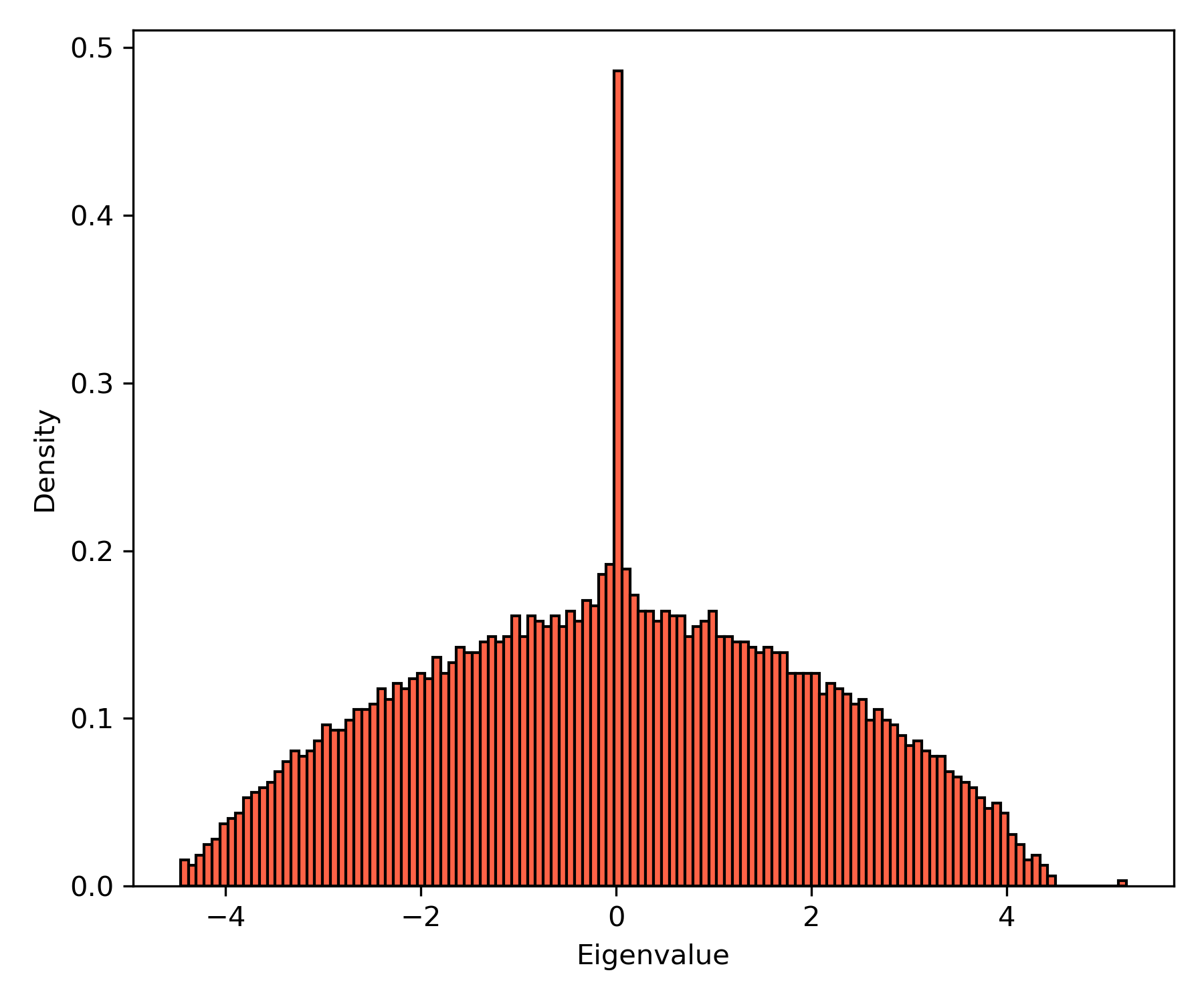}
\label{figERREG} \caption{Left: histogram of the eigenvalues of the adjacency operator of a uniformly random $4$-regular graph with $n= 4000$ vertices, in red the density of its limit \eqref{eq:kesten}. Right: histogram of the eigenvalues of the adjacency operator of an Erd\H{o}s-R\'enyi graph with $n = 4000$ vertices and edge probability $4/n$.}
\end{center}
\end{figure}

For classical random graphs sequences, there are many developments which go much beyond Theorem \ref{th:BScontspec}. They quantify the convergence of the spectrum toward its limit. We make no attempt here to present a fair account of the state of art. Let us mention that the sharpest results have been obtained for the adjacency operator of the uniform $d$-regular graphs on $n$ vertices. Notably, the typical position of the $i$-th eigenvalue (counted by non-increasing order) has been determined up to optimal order in breakthrough series of works, see \cite{zbMATH07793234,huang2025ramanujanpropertyedgeuniversality}  and references therein. 
The adjacency operator of random $d$-regular graph is invertible with high probability \cite{zbMATH07255586,zbMATH07467788}, the rank of the adjacency matrix for the Erd\H{o}s-R\'enyi ngraph is computed in \cite{MR2789584,arXiv:2303.05435}. In a different direction, when we remove the symmetric assumption, almost nothing is known about the convergence of the empirical distribution of the eigenvalues (now a probability measure on $\dC$). A remarkable exception is for the adjacency matrix of a directed Erd\H{o}s-Rényi graph on $n$ vertices with fixed average out/in degree \cite{sah2025limitingspectrallawsparse}.

\subsection{Spectral decomposition. }

Theorem \ref{th:BScontspec} calls for a general study of the spectral decomposition of the operator $A_{G,a}$ when $(G,o)$ is a unimodular random rooted graph.

For countable groups $\Gamma$ and $a \in \mathbb C[\Gamma]$ or $a \in M_n(\dC)[\Gamma]$ explicit formulas for the spectral measure $m_{a}$ are scarce, see \cite{zbMATH01402216,zbMATH07756921,zbMATH06902684}. For example, for $d \geq 2$, the spectral measure of the adjacency operator of the infinite $d$-regular tree (Cayley graph with free generators of: the free product of $\dZ_2$ with itself $d$ times, or -- for $d$ even -- the free group $\mathbb F_{d/2}$) is the Kesten-McKay measure with support $[-2\sqrt{d-1}, 2 \sqrt{d-1}]$ and density
\begin{equation}\label{eq:kesten}
m_{T_d} = \frac{d\sqrt{4 (d-1) - x^2 } }{2\pi (d^2 - x^2)} dx,
\end{equation}
see Figure \ref{figERREG}(left). For $d =2$, $T_2$ is an discrete line and $m_{T_2}$ is the arcsine distribution. The spectral density of the adjacency operator of $\dZ^d$ with its natural generators is the convolution of $m_{T_2}$ with itself $d$ times.

Interestingly, unlike these basic examples, atoms can appear for some symmetric $a \in \dC[\mathrm{F}_d]$. For lamplighter groups, the spectral measure of adjacency operators of Cayley graphs can be purely atomic \cite{MR1866850,MR2415315} or have a mixture of continuous and atomic parts \cite{zbMATH06821385} or be absolutely continuous (ac) \cite{arras2025randomschrodingeroperatorsconvolution}, see also \cite{MR1866850} and references therein. For the maximal Abelian covering of a finite graph, see  \cite{HIGUCHI2009570,li2025eigenvaluesmaximalabeliancovers}. Beyond this subtle zoology, there are few general results. A remarkable exception is \cite{zbMATH06346504}: for all Gromov-hyperbolic groups $\Gamma$ and symmetric $a \in \dR_+[\Gamma]$, the measure $m_a$ has square root type behavior  at the right edge of its support, as in \eqref{eq:kesten}.

Without surprises, for random unimodular graphs even fewer results are known. The study of the spectral decomposition of the adjacency operator of bond percolation in $\dZ^d$ (see Example \ref{exZd}) is sometimes referred as {\em quantum percolation} and was initiated by de Gennes, Lafore and Millot \cite{deGennes1959a} shortly after the celebrated work of Anderson \cite{PhysRev.109.1492} on eigenstates localization in disordered media.  It was proven in \cite{zbMATH06821385}  that the average spectral measure of the adjacency operator of bond percolation with parameter $p$ on $\dZ^2$ has a non-trivial continuous component for all $p > p_c =1/2$ (it is purely atomic for $p \leq p_c$). Proving the same claim above $p_c$ in higher dimension is an open problem. Nothing is known on eigenstates delocalization  for $d\geq 2$. 

More is known however on unimodular random trees. They are important as they appear as the BS limits of random graph ensembles. Recall from Example \ref{exER} that, with probability one, an Erd\H{o}s-R\'enyi random graph on $n$ vertices and edge probability $d/n$ converges as $n \to \infty$ to $T$, a Galton-Watson tree with offspring distribution $\POI(d)$. We illustrate  current known results on this sole example which captures essentially all possible pathological properties, see Figure \ref{figERREG} (right).  Recall that $T$ is infinite with positive probability if and only if $d >1$.

\begin{theorem}\label{th:GW}
Let $d >0$, $T$ be a Galton-Watson tree with $\POI(d)$ distribution  whose root is denoted by $o$, $A$ its adjacency operator and $m_d = \dE m_A^{\delta_o}$ its average spectral measure. There exist $1 \leq  d_1 \leq d_2$ such that 
\begin{enumerate}[label = (\roman*)]
\item (Average spectral measure: atoms) \label{GW1}
$m_d(\{\lambda \}) > 0$ if and only of $\lambda$ is a totally-real algebraic integer.  
\item (Dimension of kernel) \label{GW2} $
m_d (\{0 \}) = q + e^{-dq} + d q e^{-dq} -1
$, where $q \in (0,1)$ is the smallest root to $q = e^{-d e^{-dq}}$.

\item (Average spectral measure: continuous part) \label{GW3}
$m_d$ has a non-trivial continuous part if and only if $d >1$.

\item (Average spectral measure: ac part) \label{GW4}
$m_d$ has a non-trivial ac part for all $d > d_1$ and its total mass goes to $1$ as $d \to \infty$. Moreover, the limit
$
\lim_{t \to 0} m_d([-t,t]\backslash \{0\})/ t 
$
exists and is positive if and only if $d > e$.

\item (Spectral decomposition: atoms) \label{GW5}  For all $d> 1$, conditioned on $T$ being infinite, with probability one, there is an $\ell^2$-eigenvector with eigenvalue $\lambda$ for all totally-real algebraic integers.

\item (Spectral decomposition: ac) \label{GW6}
For all $d > d_2$, conditioned on $T$ being infinite, with probability one,  $A$ has non-trivial ac spectrum. Moreover the total mass of the ac part of  $m_A^{\delta_o}$ goes to $1$ in probability as $d \to \infty$.

\end{enumerate}

\end{theorem}

Statements \ref{GW1}-\ref{GW5} are consequences of \cite{zbMATH06411477}, see also \cite{zbMATH07174134}. The formula \ref{GW2} is established in \cite{MR2789584}. Statement \ref{GW3} is proved in \cite{zbMATH06821385}. Statements for $d >  d_i$ in \ref{GW4}-\ref{GW6} are in \cite{zbMATH07746823} while the threshold at $d =e$ in \ref{GW4} is in \cite{zbMATH07367528}. It is unclear what could possibly be the minimal values for $d_1$ and $d_2$ in statements \ref{GW3}-\ref{GW5}, see discussion in  \cite{zbMATH07746823}. The presence or not of singular continuous part in $m_d$ or in the spectrum of $A$ is a very interesting open problem. 

\subsection{Non-backtracking operator. }

We conclude this section by the introduction of a new operator which plays an important role in spectral graph theory. Let $G = ( V,E) $ be a locally finite graph. The {\em Hashimoto non-backtracting operator} of $G$ is the operator defined on compactly supported $\varphi \in \ell^2 (E)$ by the formula: 
\begin{equation}\label{eq:defB0}
B \varphi (f)  = \sum_{e \in E :  e \toNB f} \varphi(e),
\end{equation}
where $e \toNB f$ means $e_+ = f_-$ and $e \ne f^{-1}$ (the non-backtracking condition). The operator $B$ is not symmetric. 

More generally, assume that $G = ( V,E,\xi) $ is a locally finite graph and that the mark space is $\dC$. Its non-backtracking operator $B$ is defined on compactly supported $\varphi \in \ell^2 (E)$ by: 
\begin{equation}\label{eq:defB}
B \varphi (f)  = \sum_{e \in E :  e \toNB f} \xi(e) \varphi(e),
\end{equation}
We have $
B = \bar B D,
$
where $D$ is the diagonal operator $D \varphi(e) = \xi(e) \varphi(e)$ and $\bar B$ is the non-backtracking operator of $\bar G = (V,E)$.  $B$ is a bounded operator on $\ell^2(E)$ when $ \sup_{v \in V} \sum_{e \in E^+_v} |\xi(e^{-1})|^2 < \infty$.

There are correspondences between the spectra of weighted adjacency operators as defined in \eqref{eq:defAloc}-\eqref{eq:adj}  and non-backtracking operators \eqref{eq:defB}. For simplicity, let us first discuss the case of finite graphs without marks. There is a celebrated determinant identity originating from Ihara \cite{MR0223463} and later formalized in this form \cite{MR1194071,MR1749978}. 

\begin{theorem}[Ihara-Bass Formula]\label{th:IB}
Let $G = (V, E)$ be a finite graph, and let $A$, $B$ and $D$ be its adjacency, non-backtracking and degree diagianol operators ($D \varphi(v) = \deg(v) \varphi(v)$). For all $z \in \mathbb{C}$, we have:
\begin{equation}\label{eq:IharaBass}
\det(z 1_{E} - B) = (z^2 - 1)^{\chi - 1} \det(z^2 1_V - z A + D - 1_V),
\end{equation}
where $1_V$ and $1_{E}$ are the identity operators on $\mathbb{C}^V$ and $\mathbb{C}^{E}$, respectively, and $\chi = |E|/2 - |V| + 1$ is the Euler characteristic of the graph.
\end{theorem}

Ihara was developing Selberg's zeta function theory for discrete groups. The function $
\zeta(w) = \det(1_{E} - w B)^{-1}
$
is called the {\em Ihara zeta function} of the graph $G$. It admits an expression in terms of closed geodesic, see \cite{MR2768284}. 

Note that if \( G \) is a \( d \)-regular graph then \( D = d 1_V \). From Theorem~\ref{th:IB}, it easily follows that
\begin{equation}\label{eq:A2B}
\sigma(B) = \{\pm 1\} \cup \left\{ \lambda \in \mathbb{C} : \text{there exists } \mu \in \sigma(A) \text{ such that } \lambda^2 - \mu \lambda + d - 1 = 0 \right\},
\end{equation}
where \(\sigma(A)\) and \(\sigma(B)\) are the sets of eigenvalues of \( A \) and \( B \), respectively. One can also express \(\sigma(A)\) in terms of \(\sigma(B)\). There is even an explicit description of the eigenspaces of \( B \) and the corresponding eigenspaces for \( A \); see \cite{MR1749978}. Let us sketch a direct proof of \eqref{eq:A2B} and this correspondance. Suppose \( B \varphi = \lambda \varphi \) for some non-zero vector 
$\varphi \in \mathbb{C}^{E}$ and \(|\lambda| \neq 1\). We introduce the divergence vector $ \hat \varphi \in \mathbb{C}^V $ defined for every \( x \in V \) by 
$ \hat \varphi  (x) = \sum_{e  \in E_x^+} \varphi(e).$
Let $e \in E$ with $e_- = x$, $e_+ = y$, we have \( \hat \varphi (x) = (B \varphi)(e) + \varphi(e^{-1}) \). Writing \( (B\varphi)(e) = \lambda \varphi(e) \) for \( e \) and \( e^{-1} \), we arrive at the inversion fomula \( \varphi(e) = (\lambda \hat \varphi (x) -   \hat \varphi  (y)) / (\lambda^2 - 1) \). Substituting this identity into \( (B\varphi)(e) = \lambda \varphi(e) \), we find 
 $
 ( A\hat \varphi) (x) = \mu \hat \varphi (x) $ with $\lambda^2 - \mu \lambda + d - 1 = 0 $ as requested.

The Ihara-Bass identity has several extensions and variants. For example, one can write a determinant identity of the type \eqref{eq:IharaBass} relating the characteristic polynomial of \( A \), \(\det(z 1_V - A)\), to a determinant of the form \(\det(1_{E} - B D(z))\), where \( D(z) \) is an analytic diagonal operator on an unbounded domain of $\mathbb{C}$; see \cite{NIPS2009_0420,NA17}. The operator \( B D(z) \) is a weighted non-backtracking operator as in \eqref{eq:defB}. These spectral correspondences also extend to complex weights, infinite graphs and even matrix-valued weights \cite{MR4024563,MR4756991,BC23}.

\section{Edge eigenvalues of large marked graphs. }

\label{sec:edge}

We now present some results on the extreme eigenvalues of local operators of large sparse marked graphs. 

\subsection{Lower bound from BS convergence. } 
\label{subsec:LBBS}
As in Theorem \ref{th:BScontspec}, let $a : \cGrr \to \dC$ be a symmetric continuous function  with bounded support and $(G_n)$ be subsequence of finite marked graph such that $G_n$ has $n$ vertices and $(G_n) \toBS \mu$ for some $\mu \in \cP(\cGr)$. We consider the sequence of local self-adjoint operators $A_n = A_{G_n,a}$. We order the eigenvalues of $A_n$ as $\lambda_n(A_n) \leq \cdots \leq \lambda_1 (A_n)$. 
By Theorem \ref{th:BScontspec}, all but $o(n)$ eigenvalues are in a small neighborhood of $\mathrm{supp} (m_{\mu,a})$. In particular, if $k_n / n \to 0$ then 
\begin{equation}\label{eq:LBBS}
\liminf_{n \to \infty} \lambda_{k_n} (A_n) \geq  \sup \big\{ \lambda :  \lambda \in \mathrm{supp} (m_{\mu,a})\big\}.
\end{equation} 
Eigenvalues of $A_n $ which are not in a small neighborhood of $\mathrm{supp} (m_{\mu,a})$ are called the {\em spectral outliers}. They capture some geometric information on  $G_n$ which cannot be read from the local topology.

\subsection{Alon-Boppana lower bound. } 
A lower bound of the form \eqref{eq:LBBS} holds for all covering graphs under a non-negativity assumption on the local operator. We shall present this inequality in a general group setting. Let $\Gamma$ be finitely generated with unit $\ee$.  The $\ell^\infty$-norm in $\dC[\Gamma]$ is
\begin{equation*}\label{eq:defopnorm}
\| a \| = \lim_{p \to \infty} \tau \left( (a a^* )^{\frac{p}{2}} \right) ^{\frac 1 p}. 
\end{equation*}
It is equal to the operator norm of $\lambda(a)$ on $\ell^2 (\Gamma)$ which we also denote $\|\lambda(a) \|$.  The closure of $\mathbb{C}[\Gamma]$ under this norm is the reduced $C^*$-algebra of the group, denoted by $C^*_{\mathrm{red}} (\Gamma)$.

In $M_n(\dC)$ we also denote by $\| \cdot \|$ the operator norm (it is also the $\ell^\infty$-norm in $(M_n(\dC),\frac 1 n \TR)$). Let $\rho_n : \Gamma \to \mathrm{U}_n$ be a unitary representation of $\Gamma$ of dimension $n$. Recall that a subspace $H$ of $\dC^n$ is {\em invariant for $\rho_n$} if for all $g$, $\rho_n(g)  H = H$.  We say that $\rho_n$ has {\em non-negative characters} if for all $g \in \Gamma$:
$
\TR (\rho_n(g)) \geq 0.
$
This condition holds for example if $\rho_n (g) \in \mathrm{S}_n$ which is the case relevant for graphs but also if $n = m^2$ and $\rho_n (g) = \pi(g) \otimes \bar \pi (g)$ is the tensor product of a unitary representation $\pi$ of dimension $m$ (this last case is relevant in quantum channels). 

Next, we consider $a  = \sum_g a_g g \in M_k(\dC)[\Gamma]$ which has {\em non-negative joint moments} in the sense that for any $g_1,\ldots,g_l$ in $\Gamma$ and $\epsilon,\ldots,\epsilon_l$ in $\{1,*\}$, we have 
$
\TR \left( a^{\epsilon_1}_{g_1} \cdots a^{\epsilon_l}_{g_l} \right) \geq 0$. Recall finally the operators $\rho_n(a) \in M_{k}(\dC) \otimes M_n(\dC)$ and $\lambda(a) \in M_{k}(\dC) \otimes \mathcal B (\ell^2(\Gamma))$ defined by \eqref{eq:ANF}. 
\begin{proposition}[Alon-Boppana bound, non quantitative] \label{le:AB} Let $a \in M_{k}(\dC)[\Gamma]$ with support $S$ and having non-negative joint moments.  Let $\rho_n : \Gamma \to \mathrm{U}_n$ be a representation of $\Gamma$ with non-negative characters Then, for any $\epsilon  >0$, there exist $n_0 = n_0(\Gamma,S,k,\epsilon)$ and $\delta = \delta(\Gamma,S,k,\epsilon)$ such that if $n \geq n_0$ and $H_n$ is an invariant subspace of codimension at most $\delta n$, we have
$$
\| \rho_n( a) _{|\dC^{k} \otimes H_n} \| \geq ( 1 - \epsilon)\| \lambda(a) \|.
$$
Moreover, if $C^*_{\mathrm{red}} (\Gamma)$ is exact, then $n_0,\delta$ can be made uniform over all $k$ (but still depend on $(\Gamma,S,\epsilon)$).  
\end{proposition}
For the definition of exactness see \cite{zbMATH05256855} (all usual groups are exact). Note that, if $a = a^*$, by applying the lemma to $a_t = a + t \ee$ for $t$ large enough, we also obtain a lower bound on  the rightmost point in the spectra of $\rho_n( a) _{|\dC^k \otimes H_n}$.  This lemma has many variants since its original statement on adjacency operators of $d$-regular graphs \cite{MR875835} and their universal covering tree. In this case, $\|\lambda(a)\| = 2\sqrt{d -1}$ as can be seen from the support of $m_{T_d}$ in \eqref{eq:kesten}. Finite regular graphs such that the second largest eigenvalue of their adjacency matrix is below $2\sqrt{d-1}$ are called {\em Ramanujan graphs}. 
\begin{proof}[Proof of Proposition \ref{le:AB}]
Since $\| \rho_n(a) \|^2 = \| \rho_n(a) \rho_n(a)^*  \| = \| \rho_n(aa^*) \|$, we may assume that $a$ is symmetric. Let $Q$ be the orthogonal projection onto $\dC^k \otimes H_n$ and $m = \mathrm{codim}(H_n)$. For any integer $l \geq 1$,
$$
\| \rho_n( a) _{|\dC^{k} \otimes H_n} \|^{2l} = \| \rho_n( a) Q \|^{2l} = \| \rho_n( a^{2l}) Q \| \geq \frac{1}{nk} \TR ( \rho_n(a^{2l}) Q),
$$
where have used the invariance of $H_n$. Expanding the trace,
\begin{align*}
\frac{1}{nk} \TR ( \rho_n(a^{2l}) Q) & =  \frac{1}{nk} \TR ( \rho_n(a^{2l})) - \frac{1}{nk} \TR ( \rho_n(a^{2l}) (1-Q)) \\
& \geq   \frac{1} {nk} \sum_g \TR ((a^{2l} )_g ) \TR(\rho_n(g)) - \frac{m}{n} \|   \rho_n(a^{2l}) \| \\
& \geq   \frac{1} {k} \TR ((a^{2l} )_{\ee} )  - \frac{m}{n} \|  \rho_n(a^{2l}) \|,
\end{align*}
where we have used at the last line the non-negativity assumptions on $\rho_n$ and  $a$. For $b \in M_{k}(\dC) [\Gamma]$, let 
$$
\| b \|_2 = \sqrt{ \left( \frac{1}{k} \TR \otimes \langle \delta_{\ee} , \cdot \delta_{\ee}  \rangle \right) (  \lambda(bb^*)  ) }  =  \sqrt{ \sum_{g} \frac{1}{k}\TR ( b_g b_g^*) }.
$$ Note that $\frac{1} {k} \TR ((a^{2l} )_{\ee} ) = \| a^l \|_2^2$ and by Cauchy-Schwarz inequality we have, if the support of $b$ in $\Gamma$ has $s$ elements, 
$
\| \rho_n (b) \| \leq  \sum_{g} \| b_g \| \leq \sqrt{sk} \| b \|_2.
$ 
Using a rough estimation on the size of support of $a^{2l}$, we arrive at 
\begin{equation*}\label{eq:ABe}
\| \rho_n( a) _{|\dC^{k} \otimes H_n} \|^{2l} \geq \| a^l \|_2^{2} \left( 1 - \frac{|S|^l \sqrt{k} m}{n} \right).
\end{equation*}
Finally, since $ \| a^l \|_2^{1/l} \to \| \lambda(a) \|$, playing with quantifiers, we obtain the first claim. If $C^*_{\mathrm{red}} (\Gamma)$ is exact, the claim follows as in the end of proof of \cite[Theorem 7.3]{BC23}. 
\end{proof}

It is possible to be quantitative in the proof provided that we have a rate of convergence for $\| a^l \|_2^{1/l} \to \| \lambda(a) \|$. The {\em rapid decay property} is a property which allows precisely this, refer to \cite{zbMATH06859874}. The group $\Gamma$ has the rapid decay property if for all $a \in \dC[\Gamma]$ we have for some finite symmetric generated set $S$ and constants $C_1,C_2 \geq 1$, 
\begin{equation}\label{eq:defRD}
\| \lambda (a) \| \leq C_1 \mathrm{diam}^{C_2}_S (a) \| a \|_2,
\end{equation}
where $\mathrm{diam}_S(a)$ is the diameter of the support of $a$ in $\CAY(\Gamma,S)$ and $\|a\|_2^2 = \sum_g |a|^2$. This is a group property which does not depend on the choice of $S$ (only the constant $C_1$ does). All Gromov-hyperbolic groups and Abelian groups satisfies this property but not $\mathrm{SL}_d(\dZ)$ for $d \geq 3$. For $a \in M_k(\dC)[\Gamma]$, we find $\| \lambda (a) \| \leq \sqrt{k} C_1 \mathrm{diam}^{C_2}_S (a) \| a \|_2$. For the free group, the rapid decay is known as Haagerup inequality. Applying those bounds for the free group generators, we could obtain the following quantitative bound in \cite{BC23}: 
\begin{theorem} Let $(a_i)_{- d \leq i \leq d} \in M_k (\dC)^{2d+1}$ with non-negative moments. For some numerical constant $c >0$, for any  permutation matrices $(S_1,\ldots,S_d)$ of dimension $n$, setting $S_0 = 1_n$ and $S_{-i}  = S^*_i$, we have
$$
\left\| \left( \sum_{i=-d}^d  a_i \otimes S_i   \right)_{| \dC^k \otimes \IND^\perp}  \right\| \geq \left\| \sum_{i=-d}^d a_i \otimes \lambda(g_i) \right\| \left( 1 - c \frac{ \ln (2d) }{\sqrt{ \ln n}} \right),
$$ 
where $(g_1,\ldots,g_d)$ are the free generators of the free group $\mathrm{F}_d$, $g_0 = \ee$ its unit and we have set $g_{-i} = g_i^{-1}$.
\end{theorem}
The bound is universal as it does not depend on the dimension $k$ (there are better bounds for small $k$). It is stated here for a linear expression in the permutation matrices, for more general polynomials, see  \cite[Section 8]{BC23}.

\subsection{Classical random graphs. } We present here some known asymptotics for adjacency and non-backtracking operators of classical random graph ensembles.

\vspace{2pt}

\noindent{\bf Random regular graph. } It was conjectured by Alon \cite{MR875835} and proved by Friedman \cite{MR2437174} that almost all regular graphs are almost Ramanujan, see Figure \ref{figERREG} (left). For integers $n,d \geq 1$, let $\cG(n,d)$ denote the set of simple $d$-regular graphs with vertex set $V =\{1, \ldots, n\}$. For $2 \leq d \leq n-1$ and $nd$ even, $\cG(n,d)$ is non-empty. 

\begin{theorem}[Friedman's Theorem, adjacency version]\label{th:Friedman}
Let $d \geq 3$ be an integer,  $G$ a random graph distributed according to the uniform measure on $\cG(n,d)$, and let  $d = \mu_1 \geq \mu_2 \geq \cdots \geq \mu_n$ be the eigenvalues of its adjacency operator $A$. Then for every $\veps >0$, the event 
$$
 \max (\mu_2 , -\mu_n) \leq 2 \sqrt{d-1} + \veps,
$$
has probability tending to $1$ as $n$ tends to infinity (and  $nd$ is even). 
\end{theorem}

Simpler proofs now exist \cite{bordenaveCAT,CGTVH24} and the fluctuation of order $n^{-2/3}$ has been recently proved in \cite{huang2025ramanujanpropertyedgeuniversality}. In \cite{MR2437174,bordenaveCAT}, one first proves the corresponding statement for the non-backtracking matrix $B$ and then deduces it for $A$ via the Ihara–Bass formula \eqref{eq:A2B}. For a finite graph $G =(V,E)$ with $|E| = 2m$, we denote the complex eigenvalues of $B$ by $\lambda_1, \ldots, \lambda_{2m}$ indexed abritarily by their absolute values:
\begin{equation}\label{eq:eigB}
|\lambda_{2m}| \leq \cdots \leq |\lambda_2| \leq \lambda_1.
\end{equation}
If $G$ is a $d$-regular graph, $\lambda_1 = d - 1$ and the non-backtracking operator of its universal covering tree has spectral radius $\sqrt{\lambda_1}$. An equivalent formulation of Theorem \ref{th:Friedman} is the following, see Figure \ref{fig:simu} for an illustration.

\begin{theorem}[Friedman's Theorem, non-backtracking version]\label{th:FriedmanB}
Let $d \geq 3$ be an integer,  $G$ a random graph distributed according to the uniform measure on $\cG(n,d)$, and $d-1 = \lambda_1 , \lambda_2 , \ldots $ the eigenvalues of its non-backtracking operator $B$ indexed as in \eqref{eq:eigB}. Then for every $\veps >0$, the event 
$$
|\lambda_2| \leq \sqrt{\lambda_1}  + \veps,
$$
has probability tending to $1$ as $n$ tends to infinity (and  $nd$ is even). 
\end{theorem}


\vspace{2pt}

\noindent{\bf Erd\H{o}s–Rényi graph.} For fixed $d >0$, we now consider the Erd\H{o}s-R\'enyi random graph with $n$ vertices and edge probability $\min(1,d/n)$, see Example \ref{exER}.  The extreme eigenvalues of the adjacency operator $A$ diverge with $n$, and their eigenvectors are concentrated on vertices with largest degrees and their neighbors \cite{zbMATH01877117,BBK2,zbMATH08053292}.

\begin{theorem}\label{th:ERA}
Let $d > 0$, $G$ be an Erd\H{o}s–Rényi graph on $n$ vertices and edge probability $\min(1,d/n)$. Let $\mu_1 \geq  \mu_2  \ldots   \geq \mu_n$ be the eigenvalues of its adjacency operator $A$. Then, for any $\veps \in (0,1)$, with probability tending to one as $n \to \infty$, for all $1 \leq k \leq n^{1-\veps}$, we have
$$
 \ABS{ \mu_k - \sqrt{d_k} } \leq  \veps \sqrt{d_1} \quad \hbox{ and } \quad   \ABS{ \mu_{n+1-k} + \sqrt{d_k} }  \leq \veps \sqrt{d_1},
$$ 
where $d_1 \geq d_2 \geq \cdots$ is the ordered degree sequence of $G$.
\end{theorem}  
For all $1 \leq k \leq n^{1-\veps}$, we have $d_k \sim \ln (n/k) / \ln \ln (n/d)$. Theorem \ref{th:ERA} is extracted from \cite{BBK2} (which considers inhomogeneous random graphs). A sharper result has been proved in \cite{zbMATH08053292} which also implies  the localization of the top eigenvectors. As a corollary of Theorem \ref{th:ERA} there is no spectral gap in the sense that $\mu_k  / \mu_1$ converges in probability to $1$ for every $k = n^{o(1)}$ (Figure \ref{figERREG} does not show this claim).  For the non-backtracking operator $B$, however, there exists an analogue of Theorem \ref{th:FriedmanB}, proved in \cite{BLM2015}; see Figure \ref{fig:simu} for an illustration.

\begin{theorem}\label{th:ERB}
Let $d > 0$, $G$ be an Erd\H{o}s–Rényi graph on $n$ vertices and edge probability $\min(1,d/n)$. Let  {$\lambda_1 , \lambda_2 , \ldots $ the eigenvalues of its non-backtracking operator $B$ indexed as in \eqref{eq:eigB}}. Then for every $\veps >0$, the event 
$$
|\lambda_1 - d | \leq \veps \quad \text{and} \quad 
|\lambda_2| \leq \sqrt{\lambda_1}  + \veps
$$
has probability tending to $1$ as $n \to \infty$. 
\end{theorem}

This theorem is proved using much more probabilistic tools than Theorem \ref{th:FriedmanB}. For instance, the BS limit is no longer deterministic, unlike in the regular case. There are analogs of  Theorem \ref{th:ERB} for inhomogeneous random graphs  (such as the stochastic block model) and for weighted graphs, see notably \cite{BLM2015,bordenave2020detection,zbMATH07671982}. These results were motivated by the study of spectral algorithms in community detection and compressed sensing problems.

\begin{figure}[htb]\label{fig:simu}
\begin{center}\includegraphics[height = 4cm]{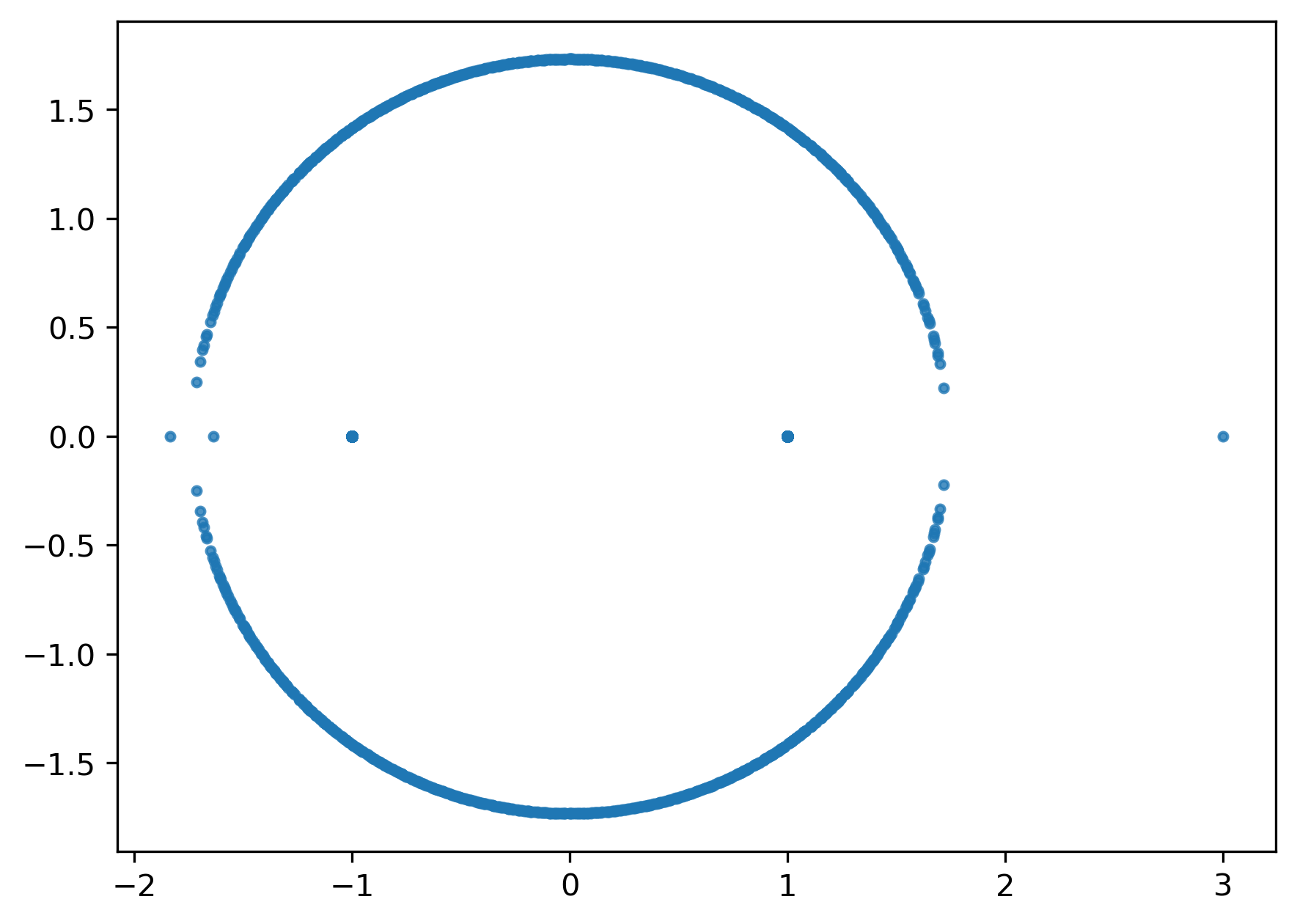} \hspace{2cm} \includegraphics[height = 4.cm]{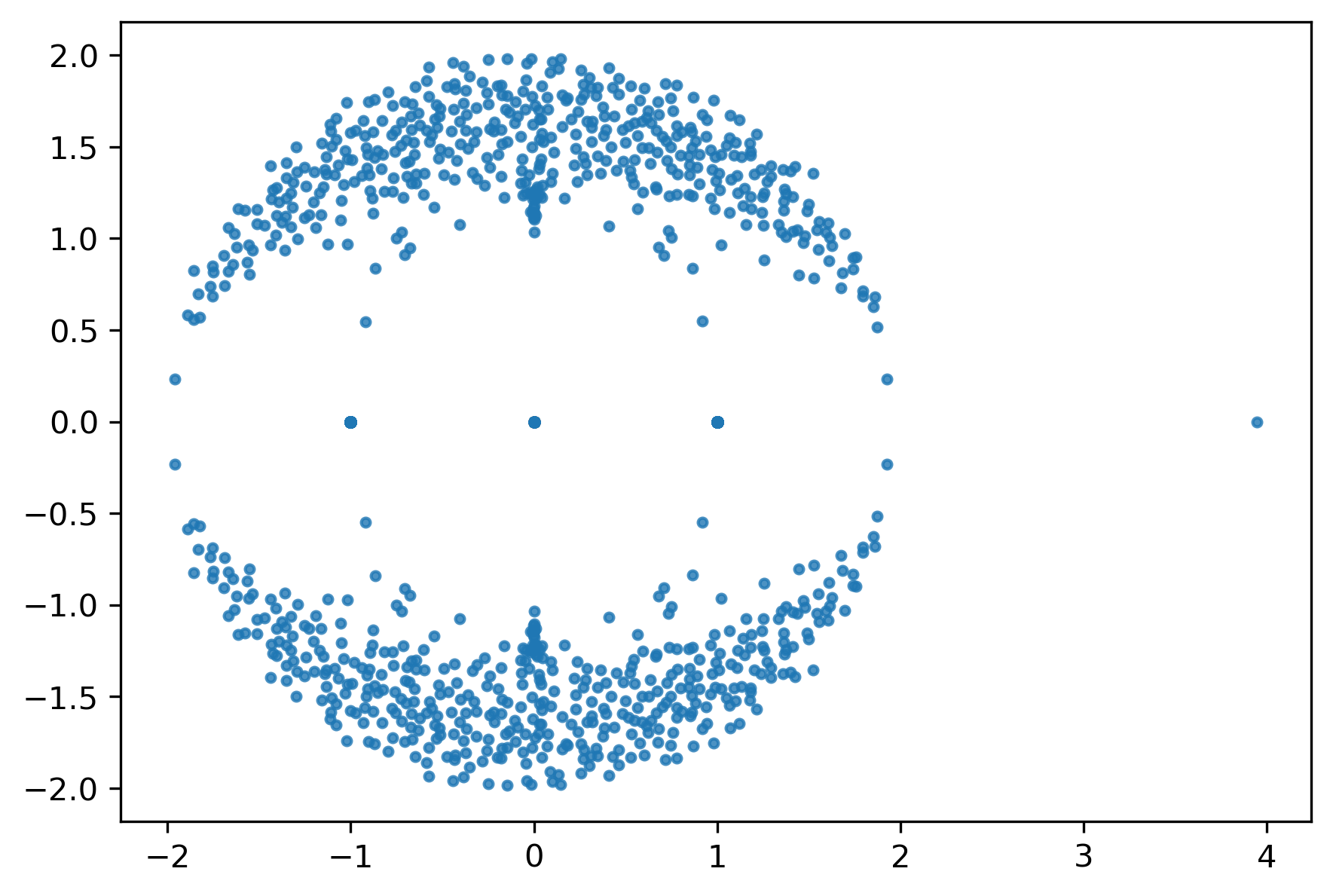}
\caption{Left: eigenvalues in the complex plane of $B$ for a uniform $d$-regular graphs with $n = 500$ vertices, $d =4$. Right: eigenvalues of $B$ for an Erd\H{o}s-R\'enyi graph with $n = 500$ vertices and edge probability $d/n$, $d=4$.}
\end{center}\end{figure}

An interesting open problem is to establish the convergence of the empirical distribution of the eigenvalues of $B$ in Theorem \ref{th:ERB}. Also, beyond adjacency and non-backtracking operators, a systematic study of the edge eigenvalues of local operators is still missing for the Erd\H{o}s-R\'enyi random graph and other inhomogeneous random graphs. As we explain next, much more is known for random Schreier and covering graphs.

\subsection{A strenghtening of BS convergence. } As usual let $\Gamma$ be finitely generated group with left regular representation $\lambda$ and $\rho_n : \Gamma \to \mathrm{U}_n$ be a subsequence of unitary representations of dimension $n$.

\vspace{2pt}

\noindent{\bf Strong convergence in distribution. } Recall from \eqref{eq:WCGr} that $\rho_n$ converges in distribution toward $\lambda$ if the normalized characters converge. We say that $\rho_n$ {\em converges strongly (in distribution)} toward $\lambda$ if, in addition, 
\begin{equation}\label{eq:SCGr}
\lim_{n \to \infty} \| \rho_n(a) \| = \| \lambda(a) \| ,\quad \hbox{for all $a \in \mathbb{C}[\Gamma]$},
\end{equation}
where $\| \cdot \|$ is the operator norm and is also the $\ell^\infty$-norm defined in \eqref{eq:defopnorm} in the relevant $*$-algebras (ie $(M_n(\dC),\frac 1 n \TR))$ and $(\dC[\Gamma],\tau)$). Strong convergence in distribution appeared in \cite{MR2183281,MR3205602} in the more general context of matrix $*$-algebras generated by finitely many matrices. We refer to the surveys \cite{magee2025,vanhandel2025}.

 For any even integer $1 \leq p <\infty$, the convergence in distribution implies the convergence of the $\ell^p$-norm $( \frac {1}{n} \TR (\rho_n(aa^*)^{p/2}))^{1/p}$  of $\rho_n(a)$,  toward the $\ell^p$-norm of $\lambda(a)$. Strong convergence allows  $p = \infty$ and it has important consequences. In particular, if $a =a^*$ then strong convergence rules out outliers: all eigenvalues of $\rho_n(a)$ are within vanishing distance from the spectrum of $\lambda(a)$.

There are various implications between convergence in distribution and strong convergence or, more precisely, between \eqref{eq:WCGr} and \eqref{eq:SCGr}. First of all, arguing as in Subsection \ref{subsec:LBBS}, \eqref{eq:WCGr} implies 
\begin{equation}\label{eq:SCGrliminf}
\liminf_{n \to \infty} \| \rho_n(a) \| \geq \| \lambda(a) \| ,\quad \hbox{for all $a \in \mathbb{C}[\Gamma]$}.
\end{equation}
Therefore, strongly converging representations are asymptotically spectrally optimal. Also, if $\Gamma$ is of subexponential growth (e.g. amenable), 
then convergence in distribution implies strong convergence. For non-amenable groups, the difficult part in establishing strong convergence is the converse inequality to \eqref{eq:SCGrliminf}:
\begin{equation}\label{eq:SCGrlimsup}
\limsup_{n \to \infty} \| \rho_n(a) \| \leq \| \lambda(a) \| ,\quad \hbox{for all $a \in \mathbb{C}[\Gamma]$}.
\end{equation} 
It even turns out that the upper bound \eqref{eq:SCGrlimsup} often implies both \eqref{eq:SCGrliminf} and \eqref{eq:WCGr}. We say that a group has the unique trace property if $C_{\mathrm{red}}^*(\Gamma)$ has a unique trace, namely the canonical trace $\tau(a) = a_e = \langle \delta_e, \lambda(a) \delta_e \rangle$. By \cite[Theorem 1.3]{zbMATH06827884} a discrete group has the unique trace property if and only if it is has no non-trivial amenable normal subgroup. If $\Gamma$ has the unique trace property, then \eqref{eq:SCGrlimsup} implies \eqref{eq:SCGrliminf} and \eqref{eq:WCGr}, see \cite[Lemma 2.13]{vanhandel2025}.



For permutation representations, we have seen in Lemma \ref{le:BSsch} that convergence in distribution is BS convergence. Strong convergence is thus a reinforcement of BS convergence. However, a clear obstruction appears. The subspace $H_n = \mathrm{span}(\mathbf{1})$ of $\mathbb{C}^{n}$ of vectors with constant coordinates is left invariant by all permutation matrices. In particular $H_n$ is an invariant subspace of $\rho_n$. Therefore, the best we can hope for is the strong convergence of the subrepresentation $(\rho_n)_{|H_n^\perp}$ toward $\lambda$: 
\begin{equation*}
\lim_{n \to \infty} \| \rho_n (a)_{| H_n^\perp} \| = \| \lambda(a) \| ,\quad \hbox{for all $a \in \mathbb{C}[\Gamma]$}.
\end{equation*}
Depending on the representation, we may sometimes aim at establishing strong convergence on the orthogonal of a larger invariant subspace considered as trivial (as in \cite[Theorem 9.3]{BC23} for a random permutation representation of Cartesian products of $\mathrm{F}_d$).

\vspace{2pt}

\noindent{\bf The linearization trick. }
There are many variants of this {\em linearization trick}. In a nutshell, it asserts that for $(V_{1}, \ldots, V_d) \in M_{k}(\mathbb C)$, we can replace the study of non-commutative $*$-polynomials $P(V_1, \ldots,V_d)$ with coefficients in $\mathbb C$ by the study of non-commutative $*$-polynomials $P(V_1, \ldots,V_d)$ with matrix-valued coefficients which are linear in the $V_i's$. We thus have traded linearity for matrix coefficients. See \cite{MR1401692,MR3585560,bordenave2020markovian} for three different instances of this trick. The most relevant for us here is due to Pisier \cite{MR1401692}: 
\begin{proposition}\label{prop:LT}
Let $S$ be a finite symmetric generating set of $\Gamma$ such that $\ee \in S$. The  following statements are equivalent: 
\begin{enumerate}[label = (\roman*)]
\item $\lim_{n} \| \rho_n(a) \| = \| \lambda(a) \|$ for all $a \in  \mathbb{C}[\Gamma]$.
\item $\lim_{n} \| \rho_n(a) \| = \| \lambda(a) \|$ for all  $k \geq 1$ and $a \in M_k(\mathbb{C})[\Gamma]$.
\item  $\lim_{n} \| \rho_n(a) \| = \| \lambda(a) \|$ for all $k \geq 1$ and $a \in M_k(\mathbb{C})[\Gamma]$ with support in $S$.
\end{enumerate}
\end{proposition}

The key implication is $(iii) \Rightarrow (i)$. It is proved by showing that for any $a \in  \mathbb{C}[\Gamma]$, there exist $k \geq 1$, $b  \in M_k(\mathbb{C})[\Gamma]$ with support in $S$ and a continuous real function $f$ such that $ \| \rho(a) \|  = f ( \| \rho(b)\|)$ for all unitary representations $\rho$ (it is enough to prove this for $\Gamma = \mathrm{F}_d$ and its free generators).   See \cite{MR1738412,BC23,vanhandel2025} for quantitative and algorithmic versions of this claim. The linearization trick is an important ingredient in most proofs of strong convergence. It is interesting to notice that elements $a \in M_k(\dC)[\Gamma]$ in (iii) are precisely the operators which appeared in \eqref{eq:ANF} as local operators of $(\Gamma,S)$-covers.

\subsection{Applications of strong convergence. }

In  their seminal paper \cite{MR2183281}, Haagerup and Thorbjørnsen have proven the almost sure strong convergence of independent random matrices of dimension $n$ sampled according to the Gaussian Unitary Ensemble (GUE). Their original motivation was to establish that  the Ext-invariant for the reduced C*-algebra of the free group on $2$ generators was not a group but only a semi-group, answering an old question in operator algebra. Since then, strong convergence has found many other interesting applications in operator algebras \cite{zbMATH07565550} but in various other fields such as in spectral geometry \cite{MR4635304} or in minimal surfaces \cite{song24} to cite a few. 
We will discuss in subsections \ref{subsec:distance}-\ref{subsec:SCFREE} two applications in graph theory. Also related to BS convergence \cite{bordenave_lacoin_2021}:  the strong convergence of permutation representations implies the cutoff for the random walks on Schreier and covering graphs of non-amenable groups  $\Gamma$ satisfying the rapid decay property \eqref{eq:defRD}.

We expect that there will be further applications of strong convergence and refer to the surveys \cite{magee2025,vanhandel2025}.  Note however that after more than twenty years of active research in the area, random matrix theory is still currently the only access to strong convergence: there are no known non-trivial deterministic examples of strong convergence of matrix algebras.

\subsection{Typical graph distance. }\label{subsec:distance}

We illustrate the strong convergence phenomenon by giving a nice yet elementary application in geometric graph theory.  As usual let $\Gamma$ be a finitely generated group with left regular representation $\lambda$. We assume that $\Gamma$ is non-amenable. If $S$ is a finite generatoring set of $\Gamma$, we denote by $\beta_S  > 0$ the exponential growth rate of $\CAY(\Gamma,S)$, that is $\beta_S = \lim_{r \to \infty} \frac 1 r \ln |B_S(r) |$, where $B_S(r)$ is the ball of radius $r$ in $\CAY(\Gamma,S)$. Let $\rho_n : \Gamma \to \mathrm{S}_{V_n}$ be a sequence of permutation representations on $|V_n|$ with $|V_n| \to \infty$. The ball of radius $r$ in $\SCH(\Gamma,S,\rho_n)$ around any point contains at most  $|B_S(r)|$ elements. It follows that the eccentry of any point is at least $(1+o(1)) \ln |V_n| / \beta_S$. More precisely,
$$
\liminf_{n \to \infty } \min_{v \in V_n} \max_{u \in V_n}  \frac{ d(u,v) }{ \ln |V_n| } \geq \frac{1}{\beta_S},
$$
where $d(u,v)$ is the graph distance in $\SCH(\Gamma,S,\rho_n)$. This bound is actually reached for almost all pairs of points if $\rho_n$ converges strongly on $\IND^\perp$ (orthogonal of vectors with constant coordinates).

\begin{lemma} Assume that $\Gamma$ is non-amenable and has the rapid decay property \eqref{eq:defRD}. Let $\rho_n : \Gamma \to \mathrm{S}_{V_n}$ be a sequence of permutation representations which converge strongly toward $\lambda$ on $\IND^\perp$. Then, for any finite generating set $S$ of $\Gamma$, for all $\veps \in (0,1)$, we have
$$
\lim_{n \to \infty } \max_{v \in V_n} \frac{ \left|\left\{ u \in V_n : d(u,v) \geq (1+\veps) \frac{ \ln |V_n| }{  \beta_S}  \right\} \right|  }{|V_n|} = 0,
$$
where $d(u,v)$ is the distance between $u$ and $v$ in $\SCH(\Gamma,S,\rho_n)$.
\end{lemma}

\begin{proof}
Fix  $0 <\beta< \beta_S$. There exists $\delta >0$ such that $\beta(1+2\delta) < \beta_S$. Then, for all $r$ large enough, we have $|B_S(r)| \geq e^{\beta(1+2\delta) r}$. For such $r$ to be defined later,  let $p = \frac 1 {|B_S(r)|} \sum_g \IND_{ g \in B_S(r)} g \in \dC[\Gamma]$, $P_n = \rho_n(p)$ and $P = \lambda ( p)$: $P_n$ and $P$ are the transition kernels of the simple random walks on $\SCH(\Gamma,S,\rho_n)$ and $\CAY(\Gamma,S)$ respectively. We have $P_n 1 = P_n^* 1 = 1 $. Thus, for any $v \in V_n$ and integer $l \geq 1$,
$$
\sum_{u} \left|P_n^l (v,u) - \frac{1}{|V_n|} \right| \leq \sqrt{|V_n|} \sqrt{ \sum_{u} \left| P_n^l (v,u) - \frac{1}{|V_n|} \right|^2} = \sqrt{|V_n|} \| (P_n^l )_{|\IND^\perp} (v,\cdot ) \|_2 \leq \sqrt{|V_n|} \|( P_n)_{|\IND^\perp} \|^l .
$$
Strong convergence implies $\|( P_n)_{|\IND^\perp} \| \leq \sqrt 2 \| P \|$ for $n$ large enough, while rapid decay \eqref{eq:defRD} implies  $\| P \| \leq C_1 r^C_2 \| p \|_2 =  C_1 r^C_2 |B_S(r)|^{-1/2} \leq (e^{-\beta(1+ \delta)}) ^{r/2}$ if $r$ was chosen large enough. We find 
$$
\sum_{u} \left|P_n^l (v,u) - \frac{1}{|V_n|} \right| \leq \left( 2 |V_n| e^{-rl \beta (1+ \delta)} \right)^{1/2}. 
$$
If $rl \geq \ln (2|V_n|) /\beta$ then (i) the right hand side goes to $0$ and (ii)  $d(u,v) >  \ln (2|V_n|) /  \beta  $ implies $P_n^l (v,u) = 0$. Use
$$
\left|\left\{ u \in V_n : d(u,v) > \frac{ \ln (2 |V_n| ) }{  \beta}  \right\} \right|   \leq  |V_n| \sum_{u} \left|\frac{1}{|V_n|} - P_n^l (v,u)  \right|,
$$
and the lemma follows since $\beta$ can be taken arbitrarily close to $\beta_S$.
\end{proof}

Remark that $S$ need not be symmetric in this lemma. Also, the proof guarantees the existence of many different paths (replace $B_S(r)$ by any other set of radius $r$ and size at least $e^{r \beta}$, e.g. an half-sphere of radius $r$). This joint use of strong convergence and rapid decay is very efficient, see \cite{bordenave_lacoin_2021} for a more elaborate use.

\subsection{Strong convergence for random representations of the free group. } 

As already mentionned, each family $(V_1,\ldots,V_d)$ of unitary matrices in $\mathrm{U}_k$ uniquely defines a unitary representation of dimension $k$ of the free group $\mathrm{F}_d = \langle g_1,\ldots, g_d \rangle$ by setting $\rho (g_i) = V_i$, for $1 \leq i \leq d$. 

Then, let $(U_1,\ldots,U_d)$ be random unitary matrices of dimension $n$ that are Haar distributed on $\mathrm{U}_n$, $\mathrm{O}_n$ or $\mathrm{S}_n$ (this last case is relevant for graphs). From what precedes this defines a random representation of $\mathrm{F_d}$ of dimension $k_n = n$ by setting $\rho_n(g_i) = V_i = U_i$. More generally, consider a unitary representation $\pi_n$ of dimension $k_n$ of $\mathrm{U}_n$, $\mathrm{O}_n$ or $\mathrm{S}_n$. Then, 
we can define a random representation of $\mathrm{F_d}$ of dimension $k_n$ by setting $\rho_n(g_i) = V_i = \pi_n(U_i)$. The above case $\pi_n (u)= u$ corresponds to the standard representation.  Below, we consider the tensor product representations defined as, for a pair of integers $(q_+,q_-)$ as
\begin{equation}\label{eq:pin}
\pi_n (u) = u^{\otimes q_+} \otimes \bar u^{\otimes q_-}.
\end{equation}
This is a representation of dimension $k_n = n^q$ with 
$
q = q_+ + q_-.
$
Note that, there is no real loss of generality as all representations of $\mathrm{U}_n$ are sub-representations of some tensor product representation.

Finally, let $H_n$ be the vector subspace of $\mathbb C^{k_n}$ of the vectors that are left invariant by $\pi_n(u)$ for all $u$ in $\mathrm{U}_N$, $\mathrm{O}_N$ or $\mathrm{S}_N$. Note that $H_n$ depends on whether we take $\mathrm{U}_n$, $\mathrm{O}_n$ or $\mathrm{S}_n$. Importantly, for $q =1$, $H_n$ is trivial for $\mathrm{U}_n$ or $\mathrm{O}_n$ while $H_n = \mathrm{span} (\IND)$ for $\mathrm{S}_n$.

\begin{theorem}\label{th1}
Let $(U_1,\ldots,U_d)$ be Haar distributed on $\mathrm{U}_n, \mathrm{O}_n$ or $\mathrm{S}_n$, and let $\rho_n$ be the random representation of $\mathrm{F}_d = \langle g_1, \ldots, g_d \rangle$ of dimension $k_n = n^q$ defined by $ \rho_n(g_i) = \pi_n(U_i)$, with $\pi_n$ as in \eqref{eq:pin}. Then, as $n \to \infty$, $\rho_n$ restricted to $H_n^\perp$ converges strongly toward $\lambda$ in probability. More precisely, for any $a \in \dC[\mathrm{F}_d]$,
\begin{equation}\label{eq:th11}
\lim_{n \to \infty} \frac{1}{k_n} \TR (\rho_n(a))  = \tau ( a )
\quad \hbox{ and } \quad 
\lim_{n\to \infty} \| \rho_n(a)_{|H_n^\perp} \|   = \| \lambda(a) \|. 
\end{equation}
\end{theorem}

The LHS of \eqref{eq:th11} is the convergence in distribution. For $q = 1$ and $\mathrm{U}_n$ or $\mathrm{O}_n$ this is due to Voiculescu \cite{MR1094052}, for general $q \geq 1$, see \cite{MR3573218}. For $\mathrm{S}_n$, the LHS of \eqref{eq:th11} is implied by \cite{MR1197059}. The RHS of \eqref{eq:th11} is the improvement to strong convergence. It was first established for $\mathrm{U}_n$ and $\mathrm{O}_n$ and $q =1$ in \cite{MR3205602}, and for general $q \geq 1$ including $q$ growing with $n$ but not too fast,  in \cite{MR4756991,MdlS2025,chen2024}.  The strongest current result is \cite{chen2024} where $q $ can be taken as large as $ n^{1/3 - \varepsilon}$ for any fixed $\varepsilon >0$. Theorem \ref{th1} for $\mathrm{S}_n$ and $q=1,2$ is first proven in \cite{MR4024563}, tensor product permutations for $q \geq 1$ are treated in \cite{CGTVH24,cassidy25}. The strongest current result being \cite{cassidy25} with $q \leq n^{1/12 - \varepsilon}$ for any fixed $\varepsilon >0$. Theorem \ref{th1} is stated in asymptotic form. There are some quantitative results, see notably \cite{BC23,chen2024}.

Let us mention a few common ingredients in the proofs of strong convergence in \cite{MR4024563,MR4756991,BC23}. The first step is to use the linearization trick (Proposition \ref{prop:LT}(iii)) with the free generators $S$ which reduces the task to prove that for any integer $k\geq 1$ and $a  = (a_i)_{-d \leq i \leq d} \in M_k(\dC)^{2d+1}$,
$$
\| (A_n)_{| \dC^k \otimes H^\perp_n} \| \leq  (1 + o(1)) \| A_{\mathrm{F}} \| \quad \hbox{with} \quad  A_n = \sum_{i=-d}^d  a_g \otimes V_i, \quad \hbox{and} \quad A_{\mathrm{F}} =\sum_{i=-d}^d  a_i \otimes \lambda(g_i),
$$
where $g_{-i} = g_i^{-1}$, $g_0 = \ee$, $U_{-i} = U_i^*$, $U_0 = 1_n$ and $V_i = \pi_n (U_i)$. The second step is to use Ihara-Bass-Selberg type formulas, as in Theorem \ref{th:IB}, to transfer spectral properties of the operator $A_n$  in terms of a family of matrix-valued non-backtracking operators (in \cite{BC23} we express directly the powers of $A_n$ in terms of sums of matrix-valued non-backtracking operators). We then estimate the trace of a high-power of such non-backtracking operator. The combinatorics is greatly simplified because non-backtracking paths in $\CAY(F_d,S)$ are trivial. The main probabilistic part is the Weingarten calculus to perform integration with the respect to Haar measures. We prove sharp upper bounds on expressions of the form, for $T$ of order up to $n^{\alpha}$ for some $0 < \alpha <1$, 
$$
\dE \left[ \prod_{t=1}^T \left( V_{x_t,y_t} - \dE  V_{x_t,y_t} \right) \right],
$$
where $U$ is Haar distributed on $\mathrm{U}_n, \mathrm{O}_n$ or $\mathrm{S}_n$, $V = \pi_n(U) = U^{\otimes q_-} \otimes \bar U^{\otimes q_+}$ and $x_t,y_t \in \{1,\ldots,n\}^q$ are coordinates. We refer to \cite{vanhandel2025} for a survey on an alternative and powerful strategy  known as the polynomial method.

\subsection{Applications to spectral graph theory. } \label{subsec:SCFREE} It is an exercise to check that if \eqref{eq:SCGr} holds then for any $a = a^* \in M_k(\dC)$, the spectrum of $\rho_n(a)$ converges to the spectrum of $\lambda(a)$ in Hausdorff distance (i.e. no outliers). As consequences of Theorem \ref{th1} for $q =1$ and $\mathrm{S}_n$, we notably obtain the following two corollaries. 
\begin{corollary}
Let $S =S^{-1}$ be a finite generating subset of $\mathrm{F}_d$ and let $A_{\mathrm F} = \lambda(1_S)$ be the adjacency operator of $\mathrm{Cay}(\mathrm{F}_d,S)$. Let $A_n = \rho_n(1_S)$ be the adjacency operator of a uniformly sampled Schreier graph on the vertex set $\{1,\ldots,n\}$ with generator $S$.
Then, in probability, the Hausdorff distance between the spectrum of $A_{\mathrm F}$ and $(A_n)_{|\mathrm 1^\perp}$ tends to $0$ as $n \to \infty$.
\end{corollary}

The case where $S$ is the set of free generators of $\mathrm{F_d}$ and their inverse is  Friedman's Theorem \cite{MR2437174}, see also Theorem \ref{th:Friedman}. The second corollary of Theorem \ref{th1} for $q =1$ and $\mathrm{S}_n$ is the generalized Alon's conjecture  \cite{MR1978881}.

\begin{corollary}
    Let $G = (V,E)$ be a finite graph, and let $G_n$ be a uniformly sampled $n$-lift of $G$. Let $A_n$ be as the adjacency operator of $G_n$ and $A_{\mathrm F} $  be the adjacency operator of the universal covering tree of $G$. Then, in probability, the Hausdorff distance between the spectrum of $A_{\mathrm F}$ and $(A_n)_{|\mathbb C^V \otimes \mathrm 1^\perp}$ tends to $0$ as $n \to \infty$.
\end{corollary}

In the two above corollaries, we can replace the adjacency operator by any other symmetric local operator.

\subsection{Beyond free group. }

Beyond random representations of the free group, instances of strong convergence in distribution are notably known for some random representations of the following groups:
{\em Cartesian products of free groups} \cite{BC23}, {\em right-angled Artin groups} \cite{magee2023stronglyconvergentunitaryrepresentations,chen2024}, {\em surface groups} \cite{MPvH2025}, {\em fully residually free groups} \cite{zbMATH07974824}.  Also, these results have also been exported to bounds on the spectral gap of the Laplace-Beltrami operator of random hyperbolic surfaces \cite{MR4635304,magee2023stronglyconvergentunitaryrepresentations,MPvH2025,hide2025spectralgappolynomialrate}, for the Weil-Petersson measure; see the breakthrough works \cite{arXiv:2304.02678,arXiv:2502.12268,hide2025spectralgappolynomialrate}.

\section*{Acknowledgments. }
This text was written during the author’s stay at the Institute for Advanced Study during the 2025-2026 academic year which was supported by the James D. Wolfensohn Fund.

\begin{spacing}{0.9}
\bibliographystyle{siamplain}
\bibliography{bib}
\end{spacing}

\end{document}